\documentclass[10pt, a4paper]{article}
\usepackage{amsmath, amssymb, amsfonts, amsthm, epsfig, enumerate, float, graphicx, sectsty}
\usepackage[pagebackref=true,colorlinks,linkcolor=red,citecolor=blue,urlcolor=blue,hypertexnames=true]{hyperref}
\usepackage[all]{xy}
\title{Short Laws for Finite Groups of Lie Type}
\author{Henry Bradford and Andreas Thom}
\date{}

\newcommand{\Addresses}{{
  \bigskip
  \footnotesize

  H.~Bradford, \textsc{Georg-August-Universit\"at G\"ottingen, 37073 G\"ottingen, Germany}\par\nopagebreak
  \textit{E-mail address}: \texttt{henry.bradford@mathematik.uni-goettingen.de}
  
  \medskip

  A.~Thom, \textsc{TU Dresden, 01062 Dresden, Germany}
  \par\nopagebreak 
  
  \vspace{-0.1cm}
  
  \textit{E-mail address}: \texttt{andreas.thom@tu-dresden.de}
}}

\newtheorem{thm}{Theorem}[section]
\newtheorem{lem}[thm]{Lemma}
\newtheorem{propn}[thm]{Proposition}
\newtheorem{coroll}[thm]{Corollary}
\newtheorem{defn}[thm]{Definition}
\newtheorem{ex}[thm]{Example}

\newtheorem{conj}[thm]{Conjecture}

\newtheorem{rmrk}[thm]{Remark}

\DeclareMathOperator{\Ad}{Ad}

\DeclareMathOperator{\charac}{char}
\DeclareMathOperator{\ccl}{ccl}
\DeclareMathOperator{\diag}{diag}
\DeclareMathOperator{\diam}{diam}

\DeclareMathOperator{\GL}{GL}
\DeclareMathOperator{\GO}{GO}

\DeclareMathOperator{\GU}{GU}
\DeclareMathOperator{\hcf}{hcf}
\DeclareMathOperator{\Hom}{Hom}
\DeclareMathOperator{\Id}{Id}

\DeclareMathOperator{\im}{im}

\DeclareMathOperator{\rk}{rk}

\DeclareMathOperator{\Alt}{Alt}
\DeclareMathOperator{\Aut}{Aut}

\DeclareMathOperator{\PGL}{PGL}
\DeclareMathOperator{\PGU}{PGU}
\DeclareMathOperator{\PSL}{PSL}

\DeclareMathOperator{\PSp}{PSp}
\DeclareMathOperator{\PSU}{PSU}
\DeclareMathOperator{\SE}{SE}
\DeclareMathOperator{\SL}{SL}
\DeclareMathOperator{\SO}{SO}
\DeclareMathOperator{\Sp}{Sp}

\DeclareMathOperator{\SU}{SU}
\DeclareMathOperator{\Sym}{Sym}

\DeclareMathOperator{\Lie}{Lie}
\DeclareMathOperator{\Out}{Out}

\begin{document}

\maketitle

\begin{abstract}
We produce new short laws in two variables 
valid in finite groups of Lie type. 
Our result improves upon 
results of Kozma and the second named author, 
and is sharp up to logarithmic factors, 
for all families except possibly the Suzuki groups. 
We also produce short laws valid for generating pairs and random pairs in finite 
groups of Lie type, and, conditional on Babai's diameter conjecture, 
make effective the dependence of our bounds on the rank. 
Our proof uses, among other tools, the Classification of Finite 
Simple Groups, Aschbacher's structure theorem for maximal subgroups 
for classical groups, and upper bounds on the diameters 
of finite simple groups due to Breuillard, Green, Guralnick, 
Pyber, Szabo and Tao. 
\end{abstract}

\tableofcontents

\section{Introduction}

A \emph{law} for a group $G$ 
is an equation holding identically in $G$. 
Every finite group satisfies a law, 
and the length of the shortest law satisfied by the finite group 
$G$ is a very natural measure of the complexity of $G$. 
In this paper we study the lengths of shortest laws 
in finite groups of Lie type. 

\subsection{Statement of Results}

Our main result is as follows. 
Let $p$ be prime and let $q$ be a power of $p$. 

\begin{thm} \label{strongmainthm}
Let $G = X (q)$ be a finite simple group of Lie type over a finite field of order $q$, where $X$ is the type of $G$. 
Then there exists a word $w_G \in F_2$ of length 
\begin{center}
$O_X (q^{a} \log (q)^{O_X(1)})$
\end{center} 
which is a law for $G$, where $a = a(X,p) \in \mathbb{N}$ 
is as in Table \ref{table:lawstable} below. 
Moreover every law for $G$ is of length $\Omega (q^{a})$ 
unless $G = {^2}B_2(q)$, in which case every law for $G$ 
is of length $\Omega (q^{1/2})$. 
\end{thm}

\begin{table}[H] 
\small
\begin{center}
\begin{tabular}{|c|c|c|c|c|}
\hline 
$X$ & $A_l$ & $^2 A_l$ & $B_l$ & $C_l$ \\ 
\hline 
$a$ & $\lfloor (l+1)/2 \rfloor$ & $\lfloor (l+1)/2 \rfloor$
& $\begin{array}{cc} 2 \lfloor l/2 \rfloor & \text{($l \geq 3$, $q$ odd)} \\ l & \text{otherwise} \end{array}$ & $l$  \\
\hline 
\end{tabular} 
\end{center}
\begin{center}
\begin{tabular}{|c|c|c|c|c|c|c|c|c|}
\hline 
$X$ & $D_l$ & $^2 D_l$ & $E_6$ & $^2 E_6$ & $E_7$ & $E_8$ & $F_4$ & $G_2$ \\ 
\hline 
$a$ & $\begin{array}{cc} l-2 & \text{($l \geq 4$ even, $q$ odd)} \\ l-1 & \text{(otherwise)}\end{array}$ & $2 \lfloor l/2 \rfloor$ & $4$ & $4$ & $7$ & $7$ & $4$ & $1$ \\
\hline 
\end{tabular} 
\end{center}
\begin{center}
\begin{tabular}{|c|c|c|c|c|}
\hline 
$X$ & $^3D_4$ & $^2B_2$ & $^2F_4$ & $^2G_2$ \\ 
\hline 
$a$ & $3$ & $1$ & $2$ & $1$ \\
\hline 
\end{tabular} 
\caption{Degree of polynomial part in lengths of shortest laws}\label{table:lawstable}
\end{center}
\end{table}

Observe that for a given $X$, 
$a$ depends only on the parity of $p$, 
and that only for groups of type $B$ or $D$. 
Hence for each fixed type $X$ (except for $^2B_2$), 
and restricting $q$ to be either odd or even, 
the asymptotic behaviour of the length of the shortest 
law holding in $X(q)$, as $q \rightarrow \infty$, 
is known up to a polylogarithmic factor. 
The fact of Suzuki groups being a difficult case 
of general statements about groups of Lie type, 
especially those with a quantitative or asymptotic flavour, 
has been observed with respect to several other problems. 
We discuss briefly some of these below. 

The integer exponent $a$ in Theorem \ref{strongmainthm} 
arises naturally in our proof in the following way: 
it is maximal such that $\PSL_2 (q^a)$ occurs as a section of $G$ 
(apart from for the Suzuki groups). 
To contextualise this bound, we may compare it with 
the minimal dimension $n$ of a projective representation of $G$. 
It turns out that $a \leq \lfloor n/2 \rfloor$ in all cases 
(see Subsection \ref{repsubsubsect} in the appendix). 
As such we have the following easy-to-state 
consequence of Theorem \ref{strongmainthm}. 

\begin{coroll} \label{mainthm}
Let $G = X (q)$ be as in Theorem \ref{strongmainthm}. 
Let $n=n(X,p)$ be the minimal dimension of a faithful projective module 
for $G$ over ${\overline{\mathbb{F}}_q}$. 
Then there exists a word $w_G \in F_2$ of length 
\begin{center}
$O_n (q^{\lfloor n/2 \rfloor} \log (q)^{O_n(1)})$
\end{center} 
which is a law for $G$. 
\end{coroll}

The exponent $\lfloor n/2 \rfloor$ in Corollary \ref{mainthm} 
agrees with the exponent $a$ from Theorem \ref{strongmainthm} 
in some cases: for instance if $X=A_l$ or $C_l$. 
In general however $a$ may be much smaller than 
$\lfloor n/2 \rfloor$. 
For comparison, the values of $n$ for each pair $(X,q)$ 
are listed in Subsection \ref{repsubsubsect}. 

The conclusion of Theorem \ref{strongmainthm} 
may easily be extended other groups closely related 
to simple groups of Lie type. 

\begin{coroll} \label{autextmainthm}
Let $G = X(q)$, $a=a(X,p)$ be as in Theorem \ref{strongmainthm}. 
Let $G \leq H \leq \Aut(G)$ and let $\hat{H}$ 
be a central extension of $H$. 
Then there exists a word $w_{\hat{H}} \in F_2$ of length 
$O_X (q^{a} \log (q)^{O_X(1)})$
which is a law for $\hat{H}$. 
\end{coroll}

Since a law for a group is also a law for any of its 
sections, the lower bounds for the length 
of the shortest law in $F_2$ satisfied by $G$ 
also applies to $\hat{H}$. 
Groups $\hat{H}$ satisfying the hypothesis of Corollary 
\ref{autextmainthm} include the isometry groups 
$\GL_n(q)$, $\GU_n(q)$, $\Sp_n(q)$ and $\GO^{\epsilon} _n(q)$ 
of the classical forms, corresponding respectively to 
$G=A_l(q)$, $G={^2}A_l(q)$, $G=C_l(q)$ and 
$G=B_l(q)$, $D_l(q)$ or $^2 D_l(q)$. 

If we consider word maps which need only vanish on generating pairs 
for groups of Lie type then we may use much shorter words. 
Define for a group $G$ and a word $w \in F_2$ the \emph{vanishing set} of $w$ in $G$ to be: 
\begin{equation} \label{vanseteqn}
Z(G,w) = \lbrace (g,h) \in G \times G : w(g,h)=1 \rbrace\text{.} 
\end{equation}
Define $c=c(X) \in \mathbb{N}$ by 
$c(^2 D_l) = (^3 D_4) = 2$ 
and $c(X)=1$ in all other cases. 

\begin{thm} \label{gensetmainthm}
Let $G = X(q)$ be as in Theorem \ref{strongmainthm}. 
Then there exists a non-trivial reduced word $w_G \in F_2$ of length
\begin{center}
$O_X (q^{c} \log (q)^{O_X(1)})$
\end{center}
such that
\begin{center}
$\lbrace (g,h) \in G : \langle g,h \rangle = G \rbrace \subseteq Z(G,w)$. 
\end{center}
\end{thm}

Theorem \ref{gensetmainthm} will be a tool in the proof of 
Theorem \ref{strongmainthm}, 
as well as being a result of interest in its own right. 

We make extensive use in the proofs of Theorems 
\ref{strongmainthm} and \ref{gensetmainthm} 
of upper bounds on the diameters of finite simple groups. 
The key conjecture in this area is due to Babai. 

\begin{conj}[\cite{BabSer}] \label{BabaiConj}
Let $G$ be a nonabelian finite simple group. Then:
\begin{center}
$\diam (G) = \log (\lvert G \rvert)^{O(1)}$.
\end{center}
\end{conj}

The definition of the diameter $\diam(G)$ of a finite group $G$ 
is deferred until Subsection \ref{diamsection}. 
Babai's Conjecture is still some way from being proven, 
in spite of remarkable progress in recent years. 
For the groups of Lie type, the state of the art is the following result. 

\begin{thm}[\cite{BGT,PySz}] \label{DiamThm}
Let $G = X(q)$ be as in Theorem \ref{strongmainthm} 
and let $n$ be as in Corollary \ref{mainthm}. Then:
\begin{center}
$\diam (G) \leq \log \lvert G \rvert ^{O_n (1)}$. 
\end{center}
\end{thm}

The dependence of the implied constant in Theorem \ref{DiamThm} upon $n$ is not given explicitly, 
and it is to be expected that the constants arising from existing proofs would be quite large. 
As such the implied constants in Theorem \ref{strongmainthm}; 
Corollary \ref{mainthm}
and Theorem \ref{gensetmainthm} are similarly inexplicit. 
Nevertheless, conditional on Babai's Conjecture, 
and using the result of \cite{KozTho} we can say a little more. 

\begin{rmrk} \label{withbab}
Assume that Conjecture \ref{BabaiConj} is true. 
Let $G = X(q)$ and $n$ be as in Corollary \ref{mainthm}, 
let $a$ be as in Theorem \ref{strongmainthm} 
and let $c(X)$ be as in Theorem \ref{gensetmainthm}. 
Then there exists a word $w_G \in F_2$ of length
\begin{center}
$O (a_n q^{a(X)} \log (q)^{O(b_n)})$
\end{center} 
which is a law for $G$, 
and a non-trivial reduced word $w^{\prime} _G \in F_2$ of length
\begin{center}
$O (c_n q^{c(X)} \log (q)^{O(d_n)})$
\end{center} 
such that
\begin{center}
$\lbrace (g,h) \in G : \langle g,h \rangle = G \rbrace \subseteq Z(G,w^{\prime})$
\end{center}
where $a_n$, $b_n$, $c_n$ and $d_n$ 
are functions of $n$ which are explicitly computable. 
\end{rmrk}

For the sake of avoiding a large amount of tedious book-keeping, 
we will not give explicit bounds on the growth of $a_n$, $b_n$, 
$c_n$ or $d_n$, and will content ourselves with remarking, 
at the appropriate points in the proofs of Theorems 
\ref{strongmainthm} and \ref{gensetmainthm}, 
where computing the dependence of the laws on $n$ is a non-trivial matter, according to the best currently known dependences. 

It is by now well-known that a generic pair of elements 
in a finite simple group of Lie type generates the group 
\cite{KanLub,LieSha,LieProb}. 
As such, the words $w_G$ arising in Theorem \ref{gensetmainthm} 
are \emph{almost laws} for $G$, 
in the sense that the probability that a random pair $(g,h)$ 
of elements of $G$ lies in $Z(G,w_G)$ tends to $1$ 
as $\lvert G \rvert \rightarrow \infty$. 
In fact, for groups of bounded rank we can do even better: 
Breuillard, Green, Guralnick and Tao \cite{BGGT} showed that 
Cayley graphs of such groups with respect to 
random pairs of generators form \emph{expanders}. 
In particular, these Cayley graphs have logarithmic diameter 
and lazy random walks on them have logarithmic mixing time. 
From this we conclude the following bound for the length 
of almost laws. 

\begin{thm} \label{randommainthm}
Let $G=X(q)$ be as in Theorem \ref{strongmainthm} 
and let $c$ be as in Theorem \ref{gensetmainthm}. 
Then there exist non-trivial reduced words $w_G \in F_2$ of length 
\begin{center}
$O_X (q ^{c} \log (q))$
\end{center}
such that, if $g , h$ are independent uniform random 
variables on $G$, 
\begin{center}
$\mathbb{P} \big[ (g,h) \in Z (G,w_G) \big] \rightarrow 1$ 
as $q \rightarrow \infty$. 
\end{center}
\end{thm}

Theorem \ref{randommainthm} complements the work of Zyrus 
\cite{Zyrus}, who produced short almost laws for $\PSL_n (q)$, 
and also provided lower bounds for the length of almost 
laws in this case. 

\subsection{Background}

The study of the structure of laws in groups is a classical 
subject, growing out of the work of Birkhoff \cite{Birk} 
in universal algebra, and further developed by many authors 
(see \cite{HNeum} and the references therein). 
At this point the behaviour of laws for finite simple groups 
was already a matter of considerable interest. 
For instance it was noted in \cite{HNeum} 
that non-isomorphic finite simple groups generate 
distinct varieties, 
and the question was posed whether there exists 
an infinite family of non-abelian finite simple groups 
satisfying a common law. 
Jones \cite{Jones} answered this question in the negative 
for the alternating groups and groups of Lie type 
and the later completion of the Classification of 
Finite Simple Groups established that this was sufficient 
to give a negative answer in general. 
Jones' result tells us that for any sequence $(G_i)_{i \in \mathbb{N}}$
of distinct finite simple groups of Lie type, 
the length of the shortest laws satisfied by $G_i$ 
tends to infinity. 
The results of the present paper address the 
rate of this divergence, emphasizing the case 
of groups of bounded rank. 

Prior to our work, the best upper bound on the length of laws for 
finite simple groups of Lie type was given by Kozma 
and the second named author. 

\begin{thm} \label{KozThomLie}
Let $G$ be a finite simple group of Lie type of Lie rank $r$, 
over a field of order $q$. 
Then there exists a word $w_G \in F_2$ of length at most
$q^{O (r)}$
which is a law for $G$. Moreover, if $G = \PSL_n (q)$, 
then $w_G$ is of length at most
\begin{center}
$\exp \big(O(n^{1/2} \log (n) )\big) q^{n-1}$. 
\end{center}
\end{thm}

Theorem \ref{KozThomLie} builds upon the work of Hadad \cite{Had}, 
and uses the Jordan decomposition to restrict the 
possibilities for the order of an element in $\PGL_n(q)$. 
The behaviour of element orders in groups of Lie type 
is a theme to which we shall return several times in what follows. 
The main result of \cite{Had} claimed a stronger upper bound, 
but the proof was found to contain a gap. 
Nevertheless, we have from \cite{Had} the following observation 
concerning lower bounds on the length of laws for 
finite simple groups of Lie type, 
which in particular shows that the exponent $\lfloor n/2 \rfloor$ 
in Corollary \ref{mainthm} is best possible. 

\begin{thm} \label{lowerboundex}
Let $k \geq 1$ and let $ w \in F_k$ be a law for $\PSL_2(q)$. 
Then $w$ has length at least $(q-1)/3$. 
We have $\PSL_2 (q^{\lfloor n/2 \rfloor})$ as a section of
 $\PSL_n (q)$ 
by restriction of scalars, so $\PSL_n (q)$ 
has no law of length less than $\Omega(q^{\lfloor n/2 \rfloor})$. 
\end{thm}

Complementary to these results for groups of Lie type, 
one may ask for short laws for the alternating 
and symmetric groups. 
The strongest available result in this direction 
is that found in \cite{KozTho}. 

\begin{thm}  \label{LawsforAlt}
There exists a law for $\Sym (n)$ of length at most: 
\begin{center}
$\exp (O(\log (n)^4 \log \log (n)))$. 
\end{center}
Further, assuming Conjecture \ref{BabaiConj} holds, 
there exists a law for $\Sym (n)$ of length at most: 
\begin{center}
$\exp (O(\log (n) \log \log (n)))$. 
\end{center}
\end{thm}

Theorem \ref{LawsforAlt} will also be useful for bounding 
the functions $a_n$, $b_n$, $c_n$ and $d_n$ 
in Remark \ref{withbab}, 
since groups of Lie type of large rank 
will contain large symmetric or alternating subgroups. 

Meanwhile, short almost laws for the symmetric groups 
and for $\PSL_n(q)$ 
were produced by Zyrus \cite{Zyrus}. 
In the latter case, lower bounds on the length 
of almost laws are also given. 

\begin{thm}[\cite{Zyrus}]
There are non-trivial reduced words $w_n \in F_2$ 
of length $$O(n^8 \log(n)^{O(1)})$$ such that 
if $g,h$ are independent uniform random variables on $\Sym(n)$ then
\begin{center}
$\mathbb{P} \big[ (g,h) \in Z (\Sym(n),w_n) \big] \rightarrow 1$
as $n \rightarrow \infty$. 
\end{center}
\end{thm}

\begin{thm}[\cite{Zyrus}]
There are non-trivial reduced words $w_{q,n} \in F_2$ 
of length $$O_n(q \log(q)^{O_n(1)})$$ such that 
if $g,h$ are independent uniform random variables on $\PSL_n(q)$ then
\begin{center}
$\mathbb{P} \big[ (g,h) \in Z (\PSL_n(q),w_{q,n}) \big] \rightarrow 1$ 
as $q \rightarrow \infty$. 
\end{center}
Further, any such words $w_{q,n} \in F_2$ 
are of length $\Omega_n(q)$. 
\end{thm}

It is expected that the methods used to prove the lower bound 
in this last result will extend to other finite simple groups 
of Lie type. This will be explored elsewhere. 

As well as being of interest in their own right, 
the existence of short laws for finite simple groups 
may be applied to the provision of short laws for other groups, and indeed of laws holding simultaneously in all 
sufficiently small finite groups. 
This latter problem is also of interest 
in geometric group theory, where it is relevant to the 
residual finiteness growth of free groups, 
originally studied by Bou-Rabee \cite{BouRab} and later in \cite{KasMat}. 
The best known result in this direction is contained in 
a previous paper of the authors \cite{BradThom}. 

\begin{thm} \label{allgrpsthm}
Let $\delta > 0$. 
For all $n \in \mathbb{N}$ 
there exists a word $w_n \in F_2$ of length
\begin{center}
$O_{\delta}(n^{2/3} \log (n)^{3+\delta})$
\end{center}
such that for every finite group $G$ satisfying $\lvert G \rvert \leq n$, 
$w_n$ is a law for $G$. 
\end{thm}

Theorem \ref{strongmainthm} in the specific cases of the groups 
$\PSL_3(q)$ and $\PSU_3(q)$ was applied in the proof of 
Theorem \ref{allgrpsthm}. 
We will also apply Theorem \ref{allgrpsthm} 
when bounding the functions $a_n$, $b_n$, $c_n$ and $d_n$ 
occurring in Remark \ref{withbab}, since many small groups 
of undetermined structure arise as subquotients of $G$, 
and we shall require 
explicit bounds on lengths of laws satisfied by these. 
It is important to stress that there is no circularity in 
our reasoning here, 
since the application of Theorem \ref{strongmainthm} 
in the proof of Theorem \ref{allgrpsthm} makes no use of the 
dependence of the implied constants in 
Theorem \ref{strongmainthm} on $n$. 

\subsection{Outline of the Proof}

Our approach to the proof of Theorem \ref{strongmainthm}
was inspired by that of Theorem \ref{LawsforAlt} in \cite{KozTho}. 
Indeed the fact that the strategy of the proof of Theorem \ref{LawsforAlt} 
could potentially provide a blueprint for 
producing short laws for groups of Lie type 
was already remarked upon in \cite{KozTho}. 
In both cases the problem is first divided 
into a search for words vanishing, respectively, 
on \emph{generating} and \emph{non-generating} pairs 
of elements in our group $G$. 

In the present setting, the generating case is precisely 
the content of Theorem \ref{gensetmainthm}. 
Producing the desired word has two stages: 
first, we identify a large subset $E$ of $G$ on which 
some short word vanishes. In most cases, 
$E$ will be the set of elements of $G$ lying in some 
maximally split maximal torus $T$ of $G$, 
so that all elements of $E$ satisfy a power law of length 
equal to the exponent $e=\exp (T)$ of $T$, 
the latter being some small-degree polynomial in $q$. 
Second, we prove the existence of a small set 
of sufficiently short non-trivial words $u_i$ 
with the following property: 
for any generating pair $g,h$ of $G$, 
there exists $i$ such that the evaluation $u_i (g,h)$ 
of $u_i$ at $(g,h)$ lies in $E$. 
From this it will follow that the vanishing sets 
of the $u_i ^e$'s cover all generating pairs, 
and combining these words by a standard commutator trick, 
we will have the required conclusion. 
It is at this stage in the argument that 
bounds on the diameter of $G$ become relevant: 
Theorem \ref{DiamThm} guarantees that the evaluation of 
a random word $u$ of length $\log \lvert G \rvert ^{O_X (1)}$ 
at a fixed generating pair $g,h$ is almost uniformly distributed 
on $G$. In particular, since $E$ contains a positive proportion 
of the elements of $G$, $u(g,h)$ lies in 
$E$ with probability bounded from below. 
It follows that if we pick our set of $u_i$ sufficiently large and independent at random, 
the desired property will hold with positive probability, 
so at least one such set must exist. 
In the setting of Theorem \ref{randommainthm} 
we have available the results of \cite{BGGT}, 
and need only take random words of length 
about $\log \lvert G \rvert$. 

In the non-generating case, we seek a word vanishing 
on all pairs in $G$ which generate a proper subgroup. 
It therefore suffices to find a law holding 
simultaneously for all maximal subgroups of $G$. 
When $G$ is a group of Lie type, 
there is a vast literature devoted to determining the 
structure of maximal subgroups of $G$, 
of which we shall avail ourselves. 

For classical groups the seminal result on the structure of 
maximal subgroups is Aschbacher's Theorem \cite{Asch}, 
which draws a dichotomy between ``geometric'' and 
``non-geometric'' subgroups. 
The geometric subgroups are those that preserve some 
extra geometric structure on the natural projective module for 
$G$: a direct-sum or tensor-product decomposition, 
for instance. They all have well-understood structure, 
being an extension built out of nilpotent groups, 
smaller groups of Lie type, and permutation groups of small 
degree. All the levels of the extension satisfy short laws 
(those for the groups of Lie type being obtained by induction) 
and from these we may easily produce a law valid on the 
whole extension. 

The non-geometric subgroups are also very restricted: 
for instance they are all central extensions 
of almost simple groups. 
We may therefore invoke the CFSG, 
and examine each family separately. 
An alternating group or a group of Lie type in characteristic 
different from that of $G$ cannot embed into $G$ unless 
it is of very small order compared to $G$: 
this follows from the work of Wagner \cite{Wagner1,Wagner2,Wagner3} 
for the alternating groups, 
and from that of Landazuri and Seitz \cite{LanSei} 
for the groups of Lie type in cross-characteristic. 
Sufficiently short laws for these groups are therefore 
easy to provide 
using Theorems \ref{LawsforAlt} and \ref{allgrpsthm}. 
This leaves us with the case of groups of Lie type 
in defining characteristic (the sporadic simple groups 
are trivial for the purposes of asymptotic statements 
such as ours). Here the possibilities for the 
embedded group are restricted thanks to the 
representation-theoretic work of Donkin \cite{Donkin} 
and Liebeck \cite{Liebeck}, 
and we have sufficiently short laws for all 
subgroups that arise by induction. 

For the exceptional groups of Lie type, 
the Aschbacher classes are not strictly defined, 
but the overall shape of the classification 
of maximal subgroups is very similar to that 
for the classical groups, as was elucidated in 
a series of papers (see \cite{LieSeiSurv} and the references 
therein, \cite{Kleid3D4,KleidG2,Malle} and the discussion 
of the Suzuki groups in \cite{BrHoRD}) 
so that we may pursue a similar strategy as in the classical case. 

The \emph{lower bounds} for the length of the shortest law for $G$ 
appearing in Theorem \ref{strongmainthm} all follow from 
Theorem \ref{lowerboundex}: 
the largest $\PSL_2$ occuring as a section of $X(q)$ 
is $\PSL_2 (q^a)$. 
The reason that our upper and lower bounds do not match 
up to a polylogarithmic factor in the case of 
the Suzuki groups $^2B_2(q)$, and in this case alone, 
is that $^2B_2(q)$ has does not have $\PSL_2$-sections 
of unbounded size as $q$ varies. 
The best available lower bound of $\Omega (q^{1/2})$ 
comes from \cite{BGTSuzuki}, and is based on \cite{Jones}. 
Roughly, the algebraic geometry of $^2B_2 (q)$ 
is sufficiently well-controlled by that of the $\Sp_4(q)$ 
in which it sits, that any law for $^2B_2 (q)$ 
of length much less than $q^{1/2}$ would also be 
a law for $\Sp_4(q)$. Since $\Sp_4(q)$ \emph{does} contain 
$\SL_2 (q)$, this is impossible. 
It is amusing to note that the Suzuki groups are outliers 
with respect to several other statements 
about groups of Lie type. 
For instance, as is well-known, 
they are the only non-abelian finite simple groups 
of order not divisible by three. 
To give a deeper example, 
Kassabov, Lubotzky and Nikolov \cite{KaLuNi} sought to 
construct generating sets with respect to which the entire 
family of finite simple groups would form an expander family. 
Alas the Suzuki groups fell outside the scope of their 
methods (also owing to the absence of large $\SL_2$ subgroups) 
and it wasn't until the later work of \cite{BGTSuzuki} 
that this gap was filled, by different methods. 

The paper is structured as followed. 
In Section \ref{PrelimSect} we specify our notation; 
introduce some preliminaries on 
laws in groups, 
diameters and mixing times for random walks on finite groups, 
and prove Theorems \ref{gensetmainthm} and \ref{randommainthm}. 
We also show how Corollary \ref{autextmainthm} 
follows from Theorem \ref{strongmainthm}. 
In Section \ref{NonGenSect} we gather results on the structure 
of maximal subgroups in finite simple groups of Lie type 
and implement our inductive argument to show that they satisfy 
short laws. 
In Section \ref{MainThmProofSect} we put everything together and 
prove Theorem \ref{strongmainthm}. 
This includes identifying subgroups which witness the lower 
bound in Theorem \ref{strongmainthm}. 
In an appendix we gather together background material 
on algebraic groups; groups of Lie type, 
and automorphisms of finite groups, 
and prove a technical result 
(Proposition \ref{orderdivpropn})
about the orders of elements in groups of Lie type. 

\section{Preliminaries and Laws for Generating Pairs} \label{PrelimSect}

\subsection{Notation}

We make use of some notations which are standard in 
the theory of finite groups: for $A$ and $B$ groups, 
$A \times B$ refers to the direct product of $A$ and $B$, 
while $A.B$ refers to an extension of undetermined structure 
with kernel $A$ and quotient $B$. 
We denote by $A \circ B$ a \emph{central product} 
of $A$ and $B$, that is a group of the form $(A \times B)/N$, 
where $N \vartriangleleft A \times B$ is the 
graph of an isomorphism between subgroups of $Z(A)$ and $Z(B)$.

For $n \in \mathbb{N}$, $n$ will also denote the cyclic 
group of order $n$. In many of the sources to which we refer, 
$[n]$ will denote a group of order $n$ of undetermined structure. 
For $G$ a group and $g \in G$, $\ccl_G(g)$ 
will denote the \emph{conjugacy class} of $g$ in $G$.  For $H < G$, $C_G(H)$ denotes the centralizer of $H$ in $G$.

We use the Dynkin notation $X(q)$ for a finite simple group of 
Lie type over a field of order $q$, where: 
\begin{center}
$X \in 
\lbrace A_l,{^2}A_l,B_l,C_l,D_l,{^2}D_l,{^3}D_4,
E_6,{^2}E_6,E_7,E_8,F_4,G_2,
{^2}B_2,{^2}G_2,{^2}F_4
 \rbrace$. 
 \end{center}
If $X$ is a twisted type, we write $X(q)$ for the group 
whose defining Frobenius automorphism has fixed 
field of order $q$. By contrast, 
many authors use $X(q)$ to denote the group whose 
natural representation is defined over a field of order $q$. 
For instance, the group that we would denote $^2 A_l (q)$, 
they would denote $^2 A_l (q^2)$. 
The (generally) simple groups of linear, 
symplectic, unitary and orthogonal type 
will also be denoted $\PSL_n (q)$, 
$\PSp_n(q)$, $\PSU_n(q)$ and $\mathrm{P}\Omega_n ^{\epsilon}(q)$ 
(for $\epsilon \in \lbrace +,-,\circ \rbrace$), 
with similar notation used to denote other groups 
in the same families in the standard fashion. 
Some sources, including \cite{BrHoRD,KleLie}, 
also use the notation $\mathbf{L}$, $\mathbf{S}$, $\mathbf{U}$
and $\mathbf{O}$ to refer to these families 
of simple classical groups, respectively. 
We avoid this convention. 

We use the \emph{Landau notation} for functions: 
for $U \subseteq \mathbb{R}$ and $f,g : U \rightarrow \mathbb{R}$ 
we write $f = O(g)$ if there exists a positive constant $C$ 
such that for all $x \in U$, 
$\lvert f(x) \rvert \leq C \lvert g(x) \rvert$. 
More generally, for 
$\lbrace f_a :U \rightarrow \mathbb{R} \rbrace_{a \in A}$ 
a family of functions, we write $f_a = O_a (g)$ 
if there exists a function $C \colon A \rightarrow (0,\infty)$ 
such that for all $a \in A$ and $x \in U$, 
$\lvert f_a(x) \rvert \leq C(a) \lvert g(x) \rvert$. 
Conversely we write $f = \Omega(g)$ 
(respectively $f_a = \Omega_a(g)$) if 
$g=O(f)$ (respectively $g=O_a(f_a)$). 
We are therefore following the (stronger) definition  
of the symbol $\Omega$ due to Knuth, 
as opposed to that of Hardy-Littlewood. 
There should be no confusion in our use of the symbols 
$O,\Omega$ for both the Landau notation and for
groups of orthogonal type, since the latter always 
appear with a superscript $+$, $-$ or $\circ$.

\subsection{Laws in Finite Groups}

\begin{defn}
Fix $x,y$ an ordered basis for the free group $F_2$ 
and let $w \in F_2 \setminus \lbrace 1 \rbrace$. 
For any group $G$ define the \emph{evaluation map} $w:G \times G \rightarrow G$ 
by $w(g,h) = \pi_{(g,h)} (w)$, 
where $\pi_{(g,h)}$ is the unique homomorphism 
$F_2 \rightarrow G$ extending $x \mapsto g$, $y \mapsto h$. 
We call $w$ a \emph{law for $G$} if $w(G \times G) = \lbrace 1_G \rbrace$. 
\end{defn}

We stress that the identity element of $F_2$ is by definition 
\emph{not} a law for any group $G$. 
We could of course more generally have defined laws 
in the free group $F_k$ of any finite rank $k \geq 1$, 
however taking an embedding $F_k \leq F_2$ 
allows us to transform any law in $F_k$ into a law in $F_2$, 
changing the length by at most a constant factor. 

\begin{ex} \label{basiclawex}
\begin{itemize}
\item[(i)] If $w$ is a law for $G$, 
then it is also a law for every subgroup and every quotient 
of $G$. 
\item[(ii)] $G$ is abelian iff $x^{-1}y^{-1}xy$ 
is a law for $G$. 
\item[(iii)] If $G$ is a finite group, 
then $x^{\lvert G \rvert}$ is a law for $G$. 
In particular $G$ satisfies some law. 
\end{itemize}
\end{ex}

We note two basic facts about the structure of laws in finite groups, 
which will enable us to construct new laws from old. 
The first allows us to combine words vanishing on subsets of a group 
to a new word vanishing on the union of those subsets, 
and is proved as Lemma 2.2 in \cite{KozTho}. 
To this end, recall for $G$ a group and $w \in F_2$ 
a word, the definition of the \emph{vanishing set} $Z(G,w)$ of $w$ on $G$ from \cite{Thom}: 
\begin{center}
$Z (G,w) = \lbrace (g,h) \in G \times G : w(g,h)=1_G \rbrace$. 
\end{center}

\begin{lem} \label{UnionLemma}
Let $w_1 , \ldots , w_m \in F_2$ be non-trivial words. 
Then there exists a non-trivial word $w \in F_2$ 
of length at most $16 m^2 \max_i \lvert w_i \rvert$ 
such that for all groups $G$, 
\begin{center}
$Z (G,w) \supseteq Z (G,w_1) \cup \ldots \cup Z (G,w_m)$. 
\end{center}
\end{lem}

\begin{ex} \label{PSL2ex} \normalfont
From Lemma \ref{UnionLemma} we may quickly prove 
the upper bound in Theorem \ref{strongmainthm} for $X = A_1$. 
It is well-known that for any $g \in \SL_2(q)$, 
the order $o(g)$ of $g$ divides one of $q-1$, $q$ or $q+1$. 
Applying Lemma \ref{UnionLemma} to the words 
$w_1 = x^{q-1}$, $w_2  =x^q$ and $w_3 = x^{q+1}$ 
we obtain a word of length $O(q)$ which is a law for 
$\SL_2(q)$, and hence also for $A_1(q) = \PSL_2(q)$. 
Note however that this approach to Theorem \ref{strongmainthm} 
already fails for $X=A_2$, since $A_2(q)=\PSL_3(q)$ 
has elements of order $\Omega (q^2)$. 
\end{ex}

Note that, as well as allowing us to increase the vanishing set of words within a single group, 
Lemma \ref{UnionLemma} allows us to take a family of groups and, given a law for each group in the family, 
produce a new law which holds in every group in the family simultaneously. For instance we have the following observation, 
which previously appeared in \cite{BouMcR}. 

\begin{ex} \label{maxorderlawex}
For $1 \leq i \leq m$ let $w_i = x^i$. 
Applying Lemma \ref{UnionLemma} to the words $w_i$, 
we obtain a non-trivial word $w \in F_2$ 
of length at most $16 m^3$, such that for every group 
$G$ satisfying: 
\begin{center}
$\max \lbrace o(g) : g \in G \rbrace \leq m$, 
\end{center}
$w$ is a law for $G$. 
\end{ex}

Relatively short laws for finite simple groups of Lie type were already constructed in \cite{Thom} 
using Example \ref{maxorderlawex} 
(since these groups do not contain elements of very large order, 
relative to their size). 
Although these laws are too weak for the conclusion 
of Theorem \ref{strongmainthm}, 
they will be useful in the proof of Theorem \ref{strongmainthm} 
nonetheless, when in the course of our induction argument 
for dealing with maximal subgroups of $G = X(q)$, 
we encounter a large number of subgroups defined over 
proper subfields of the field over which $G$ is defined. 

\begin{propn} \label{maxorderlawpropn}
Let $a(X,p)$ be as in Table \ref{table:lawstable}. 
For every $N \in \mathbb{N}$, 
there exists a word $w_{X,N} \in F_2$ of length $O (N^{6 a(X,p)})$, 
such that for every prime power $q \leq N$, 
$w_{X,N}$ is a law for $X(q)$. 
\end{propn}

\begin{proof}
Bounds on the maximal element orders of the $X(q)$ 
are given in Proposition \ref{MaxEltOrderBd}. 
Comparing Tables \ref{table:lawstable} 
and \ref{table:maxelordertable}, we have $d(X) \leq 2 a(X,p)$ 
in all cases, except for $X = D_1$ or $^2 D_1$. 
By Theorem \ref{smallrkisothm} (ii), 
$D_1(q)$ and $^2 D_1(q)$ are abelian for all $q$, 
so satisfy laws of bounded length. 
In all other cases the result 
is now immediate from Example \ref{maxorderlawex}.
\end{proof}

\begin{rmrk}
\normalfont
Although Proposition \ref{maxorderlawpropn} 
produces laws which are simultaneously valid in all sufficiently 
small groups of a fixed type $X$, 
they are longer than the analogous simultaneous 
laws arising from Theorem \ref{strongmainthm} in almost all 
cases. For we may combine by Lemma \ref{UnionLemma} 
the laws produced in Theorem \ref{strongmainthm} 
for $X(q)$ as $q$ ranges over prime powers less than $N$ 
(or only over powers of $2$ or $3$ in the cases 
$X = {^2}B_2, {^2}F_4$ or $^2G_2$). 
The laws obtained this way 
are shorter than those constructed in 
Proposition \ref{maxorderlawpropn} in all cases except $X = A_1$. 
\end{rmrk}

We also obtain from Lemma \ref{UnionLemma} 
a construction of laws for direct products of groups. 

\begin{coroll} \label{ProdLemma}
Let $G_1 , \ldots , G_m$ be groups, and suppose that for $1 \leq i \leq m$, 
$w_i \in F_2$ is a law for $G_i$. 
Then $G = G_1 \times \ldots \times G_m$ has a law of length at most 
$16 m^2 \max_i \lvert w_i \rvert$. 
\end{coroll}

\begin{proof}
Let $w$ be as in Lemma \ref{UnionLemma}. 
Then for each $i$, 
\begin{center}
$Z(G_i,w) \supseteq Z(G_i,w_i) = G_i \times G_i$. 
\end{center}
Thus, for any $g = (g_i) , h = (h_i) \in G$, 
$w(g,h) = (w(g_i,h_i)) = (1_{G_i}) = 1_G$. 
\end{proof}

Our second fact is that the length of shortest laws behaves well 
for group extensions. 
It appears (in slightly weaker form) as Lemma 2.1 in \cite{Thom}. 

\begin{lem} \label{ExtnLemma}
Let $1 \rightarrow N \rightarrow G \rightarrow Q \rightarrow 1$ be an extension of groups. 
Suppose $N, Q$ satisfy non-trivial laws in $F_2$ of length $n, m$, respectively. 
Then $G$ satisfies a non-trivial law 
in $F_2$ of length at most $nm$. 
\end{lem}

\begin{proof}
Let $w_N$, $w_Q$ be laws of minimal length for $N$, $Q$, 
respectively. We may assume both $w_N$ and $w_Q$ 
are cyclically reduced. 
Note that for any $g,h \in G$, $w_Q (g,h) \in N$. 
Suppose first that $w_Q(x,y)$ is a power of one of 
our basis elements $x$ or $y$. 
Then $w_Q(x,x)$, $w_Q(y,y)$ are both laws for $Q$, 
and freely generate 
a nonabelian free subgroup of $F(x,y)$, 
so $w_N (w_Q(x,x),w_Q(y,y))$ 
is a law for $G$ of the required length. 

If $w_Q$ is not a power of a basis element, 
then there exists a cyclic permutation $w_Q ^{\prime}$ 
of $w_Q$ such that $w_Q$, $w_Q ^{\prime}$ freely generate 
a nonabelian free subgroup of $F(x,y)$. 
Moreover $w_Q ^{\prime}$ is also a law for $Q$, 
so $w_N (w_Q,w_Q ^{\prime})$ is a law for $G$. 
\end{proof}

\begin{ex} \label{solubleex} \normalfont
Combining Lemma \ref{ExtnLemma} with Example \ref{basiclawex} (ii), we obtain for every $d \geq 1$ a non-trivial word 
$w_d \in F_2$ of length at most $4^d$ which is a 
law for every soluble group of derived length at most $d$, 
and therefore for every nilpotent group of step at most $2^d$. 
These conclusions have been improved upon by Elkasapy 
and the second author \cite{Elka,ElkaThom}. 
\end{ex}

\begin{proof}[Proof of Corollary \ref{autextmainthm}]
Let $G$ be a finite simple group of Lie type. 
Let $l$ be the length of the shortest word in $F_2$ 
which is a law for $G$. 
Let $H$ and $\hat{H}$ be as in the statement of 
Corollary \ref{autextmainthm}. 
Then $H/G \leq \Out(G)$, hence by Theorem \ref{simpleautthm}, 
$H/G$ is soluble of derived length at most $3$, 
so that by Lemma \ref{ExtnLemma} and Example \ref{solubleex}, 
$H$ satisfies a law of length at most $64 l$. 
There is an abelian normal subgroup $Z \vartriangleleft \hat{H}$ 
such that $\hat{H}/Z \cong H$, 
so that by Lemma \ref{ExtnLemma} again, 
$\hat{H}$ has a law of length at most $256 l$. 
The result now follows from Theorem \ref{strongmainthm}. 
\end{proof}

\subsection{Diameter, Expansion and Random Walks} \label{diamsection}

Let $G$ be an arbitrary finite group, 
and let $S \subseteq G$ be a generating set. 
Recall that $S$ determines a left-invariant \emph{word metric} on $G$; 
the ball about $1$ of radius $n \in \mathbb{N}$ in this metric is: 
\begin{center}
$B_S (n) = \big\lbrace s_1 \cdots s_n : 
s_1,\ldots ,s_n \in S \cup S^{-1} \cup \lbrace 1 \rbrace \big\rbrace$. 
\end{center}
The \emph{diameter of $G$ with respect to $S$} is the quantity:
\begin{center}
$\diam (G,S) = \min\lbrace n \in \mathbb{N} : B_S (n) = G \rbrace$ 
\end{center}
and the diameter of $G$ itself 
(often referred to as the \emph{worst-case diameter} of $G$) is:
\begin{center}
$\diam(G) = \max \lbrace \diam(G,S) : S \subseteq G, \langle S \rangle = G \rbrace$. 
\end{center}
We shall make use of known bounds on $\diam$ for finite simple groups, 
or more specifically the consequences of such bounds 
for random walks on such groups. 
Suppose $S$ is symmetric. 
Let $x_1 , \ldots , x_L$ be independent random variables, 
each with distribution function: 
\begin{equation} \label{distreqn}
\frac{1}{2 \lvert S \rvert} \chi_S + \frac{1}{2} \delta_{1_G}
\end{equation}
where $\chi_S$ is the indicator function of $S$ 
and $\delta_{1_G}$ is the Dirac mass at the identity, and let $\omega_L$ be the random variable on $G$ given by $\omega_L = x_1 \cdots x_L$. 
It is well-known that the number $L$ of steps taken for $\omega_L$ 
to approach the uniform distribution is controlled by 
the \emph{spectral gap} of the distribution function (\ref{distreqn}) (\cite{Lova} Theorem 5.1), 
which in turn is controlled by the diameter of $G$ with respect to $S$ 
(\cite{DiaSal} Corollary 3.1). As such we have: 

\begin{thm} \label{mixingthm}
Let $S \subseteq G$ be a symmetric generating set 
and let $E \subseteq G$.  Then: 
\begin{center}
$\mathbb{P} [\omega_L \in E] \geq \lvert E \rvert / 2 \lvert G \rvert$
\end{center}
for all $L \geq 2 \lvert S \rvert \diam (G,S)^2 \log (2 \lvert G \rvert)$. 
\end{thm}

Given this Theorem, the following is an immediate consequence of Theorem \ref{DiamThm} and, where relevant, 
the conclusion of Conjecture \ref{BabaiConj}. 

\begin{coroll} \label{rwcoroll}
Let $G = X(q)$ be as in Theorem \ref{strongmainthm}; 
let $S \subseteq G$ be a generating set 
and let $E \subseteq G$. Then: 
\begin{center}
$\mathbb{P} [\omega_L \in E] \geq \lvert E \rvert / 2 \lvert G \rvert$
\end{center}
for all $L \geq \Omega_X (\lvert S \rvert \log(q)^{\Omega_X(1)})$. 
Moreover assuming Conjecture \ref{BabaiConj}, 
the same conclusion holds for all 
$L \geq \Omega (\lvert S \rvert \log(q)^{\Omega(1)})$. 
\end{coroll}

For the conclusion of Theorem \ref{randommainthm}, 
we will need the stronger conclusion of logarithmic 
mixing time satisfied by generic generating pairs, 
which follows from the results of \cite{BGGT} 
and Theorem \ref{mixingthm}. 

\begin{thm} \label{expanderthm}
Let $G = X(q)$ be as in Theorem \ref{strongmainthm}. 
Let $g,h \in G$ be elements chosen independently uniformly 
at random. Then with probability tending to $1$ 
as $q \rightarrow \infty$, $S= \lbrace g^{\pm1},h^{\pm1} \rbrace$ 
satisfies the following conclusion. 
For any $E \subseteq G$, 
\begin{center}
$\mathbb{P} [\omega_L \in E] \geq \lvert E \rvert / 2 \lvert G \rvert$
\end{center}
for all $L \geq \Omega_X (\log(q))$. 
\end{thm}

\subsection{Laws For Generating Pairs} \label{GenSect}

We now prove Theorem \ref{gensetmainthm} and Theorem \ref{randommainthm}. 
We thereby also reduce Theorems \ref{strongmainthm} 
to known structural results on maximal subgroups 
of finite simple groups of Lie type, to be described in the following Section. The proof of Theorem \ref{strongmainthm} will be completed in Section \ref{MainThmProofSect}.

The strategy in this subsection closely mimics that employed in \cite{KozTho}. 
To wit, we identify a large subset $E \subseteq G$ 
which may be seen to lie within the vanishing set of a short word. 
Then we run simultaneously a large number of short random walks on $G$. 
Using Theorem \ref{DiamThm} (or Conjecture \ref{BabaiConj}), 
we see that with high probability, 
at least one of our random walks lands in $E$. 
It follows that \emph{as a deterministic fact}, 
there exist a set $W$ of short words (of controlled size) 
such that under any evaluation in $G$, 
some member of $W$ lies in $E$. 
We can then easily substitute the elements of $W$ 
into a word vanishing on $E$ and combine the words 
arising to obtain the required result. 

First, let us specify the set $E$. 
In \cite{KozTho}, $E \subseteq \Sym(n)$ was the set of $n$-cycles, 
so that the vanishing set of the word $x^n$ contained $E$. 
Our set $E$ will similarly satisfy 
a short power-law $x^{q^c \pm 1}$, 
where $q$ is the order of the underlying field of $G$.  
In all cases the exponent $c=c(X)$ will be as in 
Theorem \ref{gensetmainthm}. 

\begin{propn} \label{orderdivpropn}
For $G$ a finite group and $m \in \mathbb{N}$, let: 
\begin{center}
$E_G (m) = \lbrace g \in G: o (g) \mbox{ divides } m \rbrace$. 
\end{center}
Let $G = X(q)$ be as in Theorem \ref{strongmainthm}. 
Define $b(X,q) \in \mathbb{N}$ to be: 
$q+1$ for $X = {^2} A_l$ or $^2 E_6$; 
$q^2 -1$ for $X = {^2} D_l$; 
$q^2 -q+1$ for $X = {^3} D_4$, 
or $q-1$ otherwise. Then: 
\begin{equation*}
\lvert E_G \big( b(X,q) \big) \rvert  = \Omega_X (\lvert G \rvert)\text{.}
\end{equation*}
\end{propn}

We imagine that the conclusion of Proposition \ref{orderdivpropn} 
is well-known to the experts, but we have been unable to 
locate a unified proof in the literature, 
hence we present one in Appendix \ref{theappendix} below. 
For now, let us deduce Theorems \ref{gensetmainthm}
and \ref{randommainthm}  
by combining Proposition \ref{orderdivpropn} 
with, respectively, Corollary \ref{rwcoroll} 
and Theorem \ref{expanderthm}. 

\begin{proof}[Proof of Theorem \ref{gensetmainthm}]
Let $E_G = E_G \big( b(X,q) \big)$ be as in Proposition \ref{orderdivpropn}. 
Let $u_1 , \ldots , u_m$ be the results of $m$ 
independent lazy random walks of length 
$L = C_1 (\log \lvert G \rvert )^{C_1}$ on a free generating set for $F_2$, where $C_1 = C_1 (X)$ is sufficiently large 
(to be determined). 

Fix (temporarily) a generating pair $g,h \in G$. 
For each $1 \leq i \leq m$, 
the probability that $u_i (g,h) \in E_G$ is at least 
$C_2 = C_2 (X) > 0$, by Proposition \ref{orderdivpropn} 
and Corollary \ref{rwcoroll} 
(since $C_1$ is assumed to be sufficiently large, 
depending on $X$, Corollary \ref{rwcoroll} does indeed apply here). 

By independence of the $u_i$, the probability that for every 
$1 \leq i \leq m$, $u_i (g,h) \notin E_G$ is at most $(1-C_2)^m$. 
Setting $m = C_3 \log \lvert G \rvert$, 
for $C_3 = C_3 (X) > 0$ sufficiently large, 
we may take $(1-C_2)^m < \lvert G \rvert^{-2}$. 

Now, the number of possible generating pairs $(g,h)$ for $G$ 
is at most $\lvert G \rvert^2$. 
Taking a union bound over all such pairs $(g,h)$, 
the probability of the event ``for every $1 \leq i \leq m$ 
there exists a generating pair $g,h \in G$ such that $u_i (g,h) \notin E_G$'' is at most $(1-C_2)^m \lvert G \rvert^{2} < 1$. 

Therefore there exist \emph{deterministically} words 
$u_1 , \ldots u_m \in F_2$ 
with $m = C_3 \log \lvert G \rvert = O_X (\log (q))$
of length at most $L = \log \lvert G \rvert ^{O_X (1)} = O_X (\log (q)^ {O_X (1)})$ such that for every generating pair $g,h \in G$, 
there exists $1 \leq i \leq m$ such that $u_i (g,h) \in E_G$. 
Note also that, since may just as well remove $1_G$ from $E_G$ without changing the argument, 
we may assume without loss of generality that all $u_i$ 
are non-trivial in $F_2$. 
Since $F_2$ is torsion-free, $$u_1 ^{b(X,q)} , \ldots , u_m ^{b(X,q)}$$ 
are non-trivial words of length at most 
$O_X (q^{c(X)} \log (q)^{O_X (1)})$. By the definition of $E_G$, 
every generating pair $g,h \in G$ lies in $Z (G,u_i ^{b(X,q)})$ 
for some $1 \leq i \leq m$. Combining the $u_i ^{b(X,q)}$ 
by Lemma \ref{UnionLemma}, we obtain a word satisfying the conditions of Theorem \ref{gensetmainthm}, of length 
$O_X (q^{c(X)} \log (q)^{O_X (1)})$. 
\end{proof}

For the improvements required in Remark \ref{withbab}, 
we argue as before, but by assuming Conjecture \ref{BabaiConj} 
we may take walks of length $L =C_1 (\log \lvert G \rvert)^{C_1}$, 
with $C_1$ an absolute constant. 

\begin{proof}[Proof of Theorem \ref{randommainthm}]
The same proof applies, 
but we restrict to generating pairs $g,h$ 
for which the conclusion of Theorem \ref{expanderthm} applies, 
so that we may take walks of length 
$L =C_1 \log \lvert G \rvert$, with $C_1 = C_1 (X)>0$. 
\end{proof}

\section{Laws for Non-generating Pairs} \label{NonGenSect}

Having established laws valid for generating pairs in groups of Lie type in subsection \ref{GenSect}, 
we turn our attention to laws valid for \emph{non-generating} pairs. 
Trivially, if $g,h \in G$ satisfy $\langle g,h \rangle \neq G$, 
then there exists a maximal subgroup $M\lneq G$ such that $g,h\in M$. 
Therefore, our goal will be to describe the maximal subgroups of finite quasisimple groups of Lie type, 
and produce short laws which they satisfy. 

\subsection{Structure of Maximal Subgroups} 

For a given Lie type, we will identify finitely many families of subgroups, 
such that every maximal subgroup lies in at least one family, 
and produce a law valid in each family in turn. 
We will then combine the laws for the various families, using Lemma \ref{UnionLemma}. 
Crucially, the number of laws we produce in this way will depend only on the Lie type, and not on the field order $q$. 

Fortunately, there is an extensive literature on the maximal subgroups of finite simple groups, 
much of it developed in the decade following the completion of the CFSG. 
As discussed in the Introduction, 
it transpires that all maximal subgroups 
are an extension of groups for which sufficiently 
short laws are already available: 
they are of small order; nipotent of small class; 
permutation groups of small degree, 
or smaller groups of Lie type. 
For the groups of Lie type which occur, we obtain sufficiently 
short laws by invoking Theorem \ref{strongmainthm} 
for those groups and applying induction. 
To implement our induction, 
we introduce the following strict partial ordering on finite 
simple groups of Lie type: 

\begin{defn} \label{posetdefn}
Let $H = Y (p^{\mu}),G  = X (p^{\lambda})$ be groups of Lie type 
in characteristic $p$. 
We declare that $H \prec G$ if one of the following holds:
\begin{itemize}
\item[(i)] $n (Y,p) < n (X,p)$; 

\item[(ii)] $n (Y,p) = n (X,p)$, $X$ is classical 
and $Y$ is exceptional; 

\item[(iii)] $n (Y,p) = n (X,p)$, $X$ and $Y$ are exceptional and: 
\begin{itemize}
\item[(a)] Either $X=Y$ and $a^{\prime}$ is a proper divisor of $a$ 
with $a/a^{\prime}$ prime; 

\item[(b)] Or $(G,H)$ is one of 
$(E_6(p^{2a^{\prime}}),{^2}E_6(p^{a^{\prime}}))$, 
$(F_4(p^a),{^2}F_4(p^a))$ (for $p=2$) and 
$(G_2(p^a),{^2}G_2(p^a))$ (for $p=3$); 

\end{itemize}

\item[(iv)] $n (Y,p) = n (X,p)$, $X$ and $Y$ are classical and: 
\begin{itemize}
\item[(a)] Either $X \neq {^2}A_l$ and $a^{\prime}$ is a proper divisor of $a$; 

\item[(b)] Or $X = {^2}A_l$ and $a^{\prime}$ is a proper divisor of $2a$; 

\end{itemize}

\item[(v)] $n (Y,p) = n (X,p)$ and: 
\begin{itemize}
\item[(a)] $G = A_l (p^a)$ and $H$ is one of 
$B_{l^{\prime}} (p^a)$, $C_{l^{\prime}} (p^a)$, 
$D_{l^{\prime}} (p^a)$ or ${^2}D_{l^{\prime}} (p^a)$; 

\item[(b)] $p=2$, $G = C_l (p^a)$ and $H$ is one of 
$B_{l^{\prime}} (p^a)$, 
$D_{l^{\prime}} (p^a)$ or ${^2}D_{l^{\prime}} (p^a)$;

\end{itemize}
(we refer to Theorem \ref{minrepdegreethm} for the pairs of 
values $(l,l^{\prime})$ which yield the same value of $n(X,p)$ 
in (v)). 
\end{itemize}
We then extend $\prec$ to be a transitive relation. 
\end{defn}

\begin{rmrk}
\normalfont
\begin{itemize}
\item[{\it (i)}] It is straightforward to verify 
that ``$\prec$'' is a well-defined strict partial ordering. 

\item[{\it (ii)}] It is especially important to note that the 
following relations hold, as consequences of the above. 
\begin{align*}
\PSL_n(q) & \prec \PSU_n(q) \text{ (by (iv)(b));} \\ 
\PSp_n(q) & \prec \PSU_n(q) \text{ (by (v)(a) and transitivity);} \\
\PSU_n (q) & \prec \PSL_n(q^2) \text{ (by (iv)(a)). }
\end{align*}
\item[{\it (iii)}] As we shall see, 
if $H$ is a simple group of Lie type 
in the same characteristic as $G$, 
which arises as a proper section of $G$, 
then $H \prec G$. 
The converse does not hold, 
however it will be much easier in practice to 
work with ``$\prec$'' than to attempt to 
perform induction on the family of sections directly. 
\end{itemize}
\end{rmrk}

\subsection{Geometric Subgroups of Classical Groups}

For the classical groups, 
the key result on the structure of maximal subgroups is Aschbacher's Theorem \cite{Asch}. 
Aschbacher's paper defines eight classes of subgroups, denoted $\mathcal{C}_1$ - $\mathcal{C}_8$, 
and known collectively as the \textquotedblleft geometric subgroups\textquotedblright. 
We shall not define these classes precisely, 
so suffice it to say that each type is associated with some extra geometric structure 
on the natural module associated with our group. 
Aschbacher's Theorem asserts that every maximal subgroup of a quasisimple classical group 
either belongs to one of the classes $\mathcal{C}_i$, 
or belongs to the class of \textquotedblleft non-geometric subgroups\textquotedblright, denoted $\mathcal{S}$. 
Moreover, every subgroup in $\mathcal{S}$ is an almost simple group 
satisfying certain additional irreducibility conditions. 

We make no claims as to the disjointness of the families of subgroups described below, 
or that every subgroup we consider is indeed maximal in the corresponding classical group. 
All that we require is that every maximal subgroup appears at least once in one of the nine classes. 

Before stating the structure theorem for maximal subgroups 
of classical groups we introduce some additional terminology. 
Let $n = n(X,p)$ be as in Corollary \ref{mainthm}. 

\begin{defn} \label{labeltermdefn}
Let $G = X(q)$ be a finite simple group of Lie type, 
with $q = p^{\lambda}$, $p$ prime. 
Let $S$ be a section of $G$. 
\begin{itemize}
\item[(i)] $S$ is a \emph{Lie-like level for $G$} 
if there exist $m,n_i,\lambda_i \in \mathbb{N}_{>0}$ with:
\begin{equation} \label{Lielikeineq}
\sum_{i=1} ^m n_i \lambda_i \leq n(X,p); 
\end{equation}
groups of Lie type 
$G_1 = Y_1 (q^{\lambda_1}),\ldots,G_m = Y_m (q^{\lambda_m})$ 
such that $n_i = n(Y_i,p)$, $\lambda_i a(Y_i,p) \leq a(X,p)$ 
and $G_i \prec G$, 
and a finite abelian group $A$ 
such that $S$ is a quotient of: 
\begin{center}
$L=A \times G_1 \times \cdots \times G_m$.
\end{center}

\item[(ii)] $S$ is a \emph{subfield level for $G$} 
if either (a) there is a group $H = Y (p^{\mu})$ 
of Lie type with $S \cong H$
such that $\mu$ is a proper divisor of $\lambda$, 
with $\lambda/\mu$ prime and $X=Y$ or $(X,Y)=(D_l,{^2}D_l)$, 
or (b) $G = \PSU_n (q)$ and 
$S = \Sp_n(q)$ (for $n$ even) or $\SO_n ^{\epsilon}(q)$. 

\item[(iii)] $S$ is a \emph{$p$-level for $G$} 
if it is a $p$-group. 

\end{itemize}
\end{defn}

\begin{rmrk} \label{labeltermrmrk}
The number of possibilities for $L/A$ 
is bounded by a function of $n$ alone 
where $L$, $A$ are as in Definition \ref{labeltermdefn} (i). 
\end{rmrk}

\begin{thm} \label{structurethmclassical}
Let $M$ be a maximal subgroup of $G$. 
Then either (geometric type) there exists a subnormal series: 
\begin{center}
$M = K_1 \vartriangleright K_2 \vartriangleright K_3 
\vartriangleright K_4 \vartriangleright K_5 = 1$
\end{center}
such that one of the following holds: 
\begin{itemize}
\item[\textbf{$\mathcal{C}_1$:}] 
$K_4=1$, $K_3/K_4$ is a $p$-level for $G$, 
$K_2/K_3$ is a Lie-like level for $G$, and $K_1/K_2$ is abelian; 
\item[\textbf{$\mathcal{C}_2$:}] 
$K_4=1$, $K_3/K_4$ is a Lie-like level for $G$, 
$K_2/K_3$ is abelian, and $K_1/K_2$ is a subgroup of $\Sym (n)$; 
\item[\textbf{$\mathcal{C}_3$:}] 
$K_4=1$, $K_3/K_4$ is a Lie-like level for $G$, 
$K_2/K_3$ is abelian, and $\lvert K_1/K_2 \rvert \leq 2n$; 
\item[\textbf{$\mathcal{C}_4$:}] 
$K_3 = 1$, $K_2/K_3$ is a Lie-like level for $G$, 
and $\lvert K_1/K_2 \rvert \leq n$; 
\item[\textbf{$\mathcal{C}_5$:}] 
$K_3 = 1$, $K_2/K_3$ is a subfield level for $G$, 
and $\lvert K_1/K_2 \rvert \leq n$; 
\item[\textbf{$\mathcal{C}_6$:}] 
$K_3 = 1$, $K_2/K_3$ is a $2$-step nilpotent group, 
and $K_1/K_2$ is a subgroup of $\Sym (n^2)$; 
\item[\textbf{$\mathcal{C}_7$:}] 
$K_4$ is abelian, $K_3/K_4$ is a Lie-like level for $G$, 
$\lvert K_2/K_3 \rvert \leq n^2$, 
and $K_1/K_2$ is a subgroup of $\Sym (n)$; 
\item[\textbf{$\mathcal{C}_8$:}] 
$K_4 = 1$, $K_3 / K_4$ is abelian, 
$K_2 / K_3$ is a Lie-like level for $G$, 
and $K_1/K_2$ is abelian;
\end{itemize}
or (non-geometric type or type $\mathcal{S}$) 
there is a non-abelian finite simple group $S$ 
such that $S \leq M \leq \Aut(S)$ 
and the preimage $\hat{S}$ of $S$ in the full covering group 
$\hat{G}$ of $G$ is absolutely irreducible on the 
natural module for $\hat{G}$. 
\end{thm}

\begin{proof}
This is immediate from Theorem 2.2.19 in \cite{BrHoRD}, 
which in turn is based on the Main Theorem of \cite{KleLie}. 
Class $\mathcal{S}$ is described in Definition 2.1.3 of \cite{BrHoRD}, 
and the structure of the subnormal series for the groups in cases 
$\mathcal{C}_1$-$\mathcal{C}_8$ follows from Tables 2.3, 2.5-2.11 in \cite{BrHoRD}, noting that in case $\mathcal{C}_6$, 
$n = r^m$ for $r$ prime, so that the natural module for $K_1/K_2 = \Sp_{2m} (r)$, $\SO_{2m} ^{\pm} (r)$ or $\Omega_{2m} ^{\pm} (r)$ is a set of order $r^{2m} = n^2$ on which $K_1/K_2$ acts faithfully. 
\end{proof}

Recall the strict partial order $\prec$ introduced in Definition \ref{posetdefn}. 

\begin{propn} \label{geompropn}
Assume that all finite simple groups of Lie type $H$ 
with $H \prec G$ satisfy a law as in Theorem \ref{strongmainthm}. 
Then there is a word $w_{\rm geom} \in F_2$ 
of length $O_X (q^{a(X,p)} \log (q)^{O_X(1)})$ 
such that, if $M \leq G$ is a geometric maximal subgroup 
as in cases $\mathcal{C}_1$-$\mathcal{C}_8$ of Theorem 
\ref{structurethmclassical}, 
then $w_{\rm geom}$ is a law for $M$. 
\end{propn}

The following will be used to deal with the $p$-levels. 

\begin{lem} \label{nilpSylGLlem}
Suppose $q$ is a power of the prime $p$, and let $P$ 
be a Sylow $p$-subgroup of $\PGL_n(q)$. 
Then $P$ is nilpotent of class at most $n-1$. 
\end{lem}

\begin{coroll} \label{uniplawcoroll}
There is a non-trivial word in $F_2$, 
of length depending only on $n$, 
which is a law for every $p$-subgroup of $\PGL_n (q)$. 
\end{coroll}

\begin{proof}
This follows from Lemma \ref{nilpSylGLlem} 
and Example \ref{solubleex}. 
\end{proof}

\begin{proof}[Proof of Proposition \ref{geompropn}]
By Lemmas \ref{UnionLemma} and \ref{ExtnLemma}, 
it will suffice to provide a law of appropriate length 
for each of the factors 
$K_i / K_{i+1}$ in the subnormal series for $M$ 
in each of the cases $\mathcal{C}_1$-$\mathcal{C}_8$. 

Abelian factors and the $2$-step nilpotent factor from case 
$\mathcal{C}_6$ satisfy laws of bounded length. 
The $p$-level from case $\mathcal{C}_1$ 
is dealt with by Corollary \ref{uniplawcoroll} 
(with the length of the law obtained as in 
Example \ref{solubleex}). 
Those factors of order bounded by a polynomial function of $n$ 
(occuring in cases $\mathcal{C}_3$, $\mathcal{C}_4$, 
$\mathcal{C}_5$ and $\mathcal{C}_7$) 
are handled by Theorem \ref{allgrpsthm}. 
Those factors embedding into $\Sym (n)$ 
(cases $\mathcal{C}_2$ and $\mathcal{C}_7$) 
or $\Sym (n^2)$ (case $\mathcal{C}_6$) 
are handled by Theorem \ref{LawsforAlt} 
(with the second, stronger bound being available 
in the setting of Remark \ref{withbab}). 

We are left with the groups occuring as Lie-like or 
subfield levels for $G$ 
(these being the only cases in which the induction hypothesis is actually used). 

First suppose $S$ is a Lie-like level for $G$. 
Let $L$, $A$, $m$, $G_i$, $n_i$ and $\lambda_i$ be as in 
Definition \ref{labeltermdefn} (i). 
By Lemma \ref{ExtnLemma} it suffices to produce a short law valid 
simultaneously in all possible $L/A$. 
Indeed by Remark \ref{labeltermrmrk} and Lemma \ref{UnionLemma}, 
it suffices to produce a short law for each 
of the possible groups $L/A$ individually. 
By induction, each $G_i$ satisfies a law of length: 
\begin{center}
 $O_{X_i} (q^{\lambda_i a(X_i,p)} (\lambda_i \log (q))^{O_{X_i}(1)}) 
 = O_X (q^{a(X,p)} \log (q)^{O_X(1)})$
\end{center}
(since by (\ref{Lielikeineq}) all $\lambda_i \leq n$ and either 
all $n_i < n$ or $m=1$ and $\lambda_1 = 1$). 
By Lemma \ref{ExtnLemma}, $L$ satisfies a law of length at most: 
\begin{center}
$O_X (q^{a(X,p)} \log (q)^{O_X(1)})$. 
\end{center}
Therefore suppose $S$ is a subfield level for $G$. 
In case (b) of Definition \ref{labeltermdefn} (ii), 
$S$ satisfies a sufficiently short law 
by induction (using Definition \ref{posetdefn} (iv)(b)). 
Suppose therefore that we are in case (a) 
of Definition \ref{labeltermdefn} (ii).  
Let $d=\lambda/\mu$ be the degree of the extension of the field 
over which $G$ is defined, over the field over which 
$S$ is defined. 
By Proposition \ref{maxorderlawpropn}, there exists a 
word $w_{small}$ of acceptable length which is a 
law for all $S$ such that $d \geq 7$. 
Meanwhile for $d=2,3$ or $5$ there exists by induction 
a word $w_d$ of acceptable length which is a law for $S$. 
Combining these laws by Lemma \ref{UnionLemma} 
we have a word of acceptable length which is a law 
for all subfield levels. 
\end{proof}

\subsection{Non-geometric Subgroups of Classical Groups}

Recall that, by Aschbacher's Theorem, 
if $M$ is a maximal subgroup of $G$ not lying in any of the classes 
$\mathcal{C}_1$-$\mathcal{C}_8$ above 
(that is, $M$ is a \emph{non-geometric} subgroup), 
there is a non-abelian finite simple group $S$ 
such that $S \leq M \leq \Aut(S)$ 
and the full covering group $\tilde{S}$ of $S$ 
is absolutely irreducible on the 
natural module $V$ for $G$. 
In this subsection we show that there is a short law 
satisfied by all such $M$. 

\begin{propn}\label{nongeomprop}
Assume that all finite simple groups of Lie type $H$ 
with $H \prec G$ satisfy a law as in Theorem \ref{strongmainthm}. 
Then there is a word $w_{\rm nongeom} \in F_2$ 
of length $O_X (q^{a(X,p)} \log (q)^{O_X(1)})$ 
such that, if $M \leq G$ is a non-geometric maximal subgroup 
as in cases $\mathcal{S}$ of Theorem 
\ref{structurethmclassical}, 
then $w_{\rm nongeom}$ is a law for $M$. 
\end{propn}

As always, by Lemma \ref{UnionLemma} 
it suffices to produce boundedly many sufficiently 
short laws such that each $M$ satisfies one of them. 
First, by the CFSG we may exclude the case of $S$ sporadic: 
there are finitely many possibilities for the corresponding $M$, 
so $M$ satisfies a law of \emph{bounded} length. 
For the same reason, we may omit 
from consideration an arbitrarily large bounded 
number of other possibilities for $S$. 

Second, $S \vartriangleleft M$, so we may naturally identify $M$ with a subgroup of $\Aut (S)$, 
and $M/S$ with a subgroup of $\Out (S)$. 
Thus $M/S$ is soluble of derived length at most $3$, 
so applying Lemma \ref{ExtnLemma} and Example \ref{solubleex}, 
it suffices to find a short law for $S$ 
(note that we are not using the full power of the 
Schreier hypothesis here, since we already excluded sporadic 
groups). 
Let $\Lie (p)$ be the class of all finite simple groups 
of Lie type in characteristic $p$. 
We deal separately with the three cases of $S$ an alternating group; $S$ a group of Lie type but $S \notin \Lie (p)$, 
and $S \in \Lie (p)$. 

The following standard observation 
(following from Schur's Lemma) will allow us to move 
between linear representations of $\tilde{S}$ 
and projective representations of $S$. 

\begin{lem} \label{BurnsideLemma}
Let $G$ be a group, and let $\rho$ be an absolutely irreducible 
linear representation of $G$. 
Then $\rho(Z(G))$ consists of scalar matrices. 
\end{lem}

Consider first the case that $S \cong \Alt(m)$ is alternating. 
The work of Wagner allows us to bound $m$ in terms of the dimension $n$ of $V$ alone. 
In particular 
there is a law of length depending only on $n$ which is satisfied simultaneously by all such $S$. 

\begin{thm} \label{WagnerThm}
Let $m \geq 9$ and let $\rho$ be a non-trivial absolutely 
irreducible representation of the full covering group 
$\widetilde{\Alt}(m)$ of $\Alt(m)$. 
over an arbitrary field $\mathbb{F}$. 
\begin{itemize}
\item[(i)] If $\rho (Z (\widetilde{\Alt}(m))) \neq 1$, then $\charac (\mathbb{F}) \neq 2$, 
and $\dim (\rho) \geq 2^{\lfloor (m-s-1)/2 \rfloor}$, 
where there are non-negative integers $w_1 > \ldots > w_s$ such that:
\begin{center}
$m = 2^{w_1} + \ldots + 2^{w_s}$. 
\end{center}
In particular, for any $\epsilon > 0$, 
$\dim (\rho) \geq 2^{(1-\epsilon)m/2}$ for $m$ sufficiently large. 

\item[(ii)] If $\rho (Z (\widetilde{\Alt}(m))) = 1$ and $\charac (\mathbb{F}) \neq 0$, 
then $\dim (\rho) \geq m-2$. 

\end{itemize}
\end{thm}

\begin{proof}
If the hypothesis of (i) holds, then Theorem 1.3 of \cite{Wagner1} 
applies, as the induced projective representation of $\Alt(m)$ 
is \emph{proper}, in the sense of \cite{Wagner1}. 

In the setting of (ii), $\rho$ descends to a linear representation 
of $\Alt(m)$ then we are done by 
Theorem 1.1 of \cite{Wagner2} or Theorem 1.1 of \cite{Wagner3}. 
\end{proof}

The bound for the length of the law satisfied concurrently 
by the alternating $S$ arising in this case, 
required for Remark \ref{withbab}, follows straightforwardly 
from the dimension bounds in 
Theorem \ref{WagnerThm}, Theorem \ref{LawsforAlt} 
(recalling that Conjecture \ref{BabaiConj} is assumed in 
Remark \ref{withbab}) and the fact that $(\Alt(m))_m$ 
is an ascending nested sequence. 

Second, suppose $S$ to be a simple group of Lie type with $S \notin \Lie (p)$. Lower bounds on the dimensions of cross-characteristic representations of finite simple groups of Lie type are provided by the work of Landazuri and Seitz \cite{LanSei}. 
From their very detailed results, we require only the following, 
which may be read off directly from the table 
in the main result of \cite{LanSei}. 

\begin{thm} \label{LandSeitThm}
There exist explicit absolute constants $c_1 , c_2 > 0$ 
such that the following holds. 
Let $S=Y(r)$ be a finite simple group of Lie type 
over a finite field of order $r$, 
and let $m$ be the minimal dimension of a 
faithful projective representation of $S$ 
over $\overline{\mathbb{F}_r}$. 
Let $\rho$ be a non-trivial projective representation of 
$S$ over $\mathbb{F}_q$. 
If $(r,q)=1$ then: 
\begin{center}
$\dim (\rho) \geq c_1 r^{c_2 m}$. 
\end{center}
\end{thm}

We use Theorem \ref{LandSeitThm} to bound the order of $S$ 
by an explicit function of $n=n(X,p)$ alone. 
All such $S$ will therefore satisfy a sufficiently short law 
by Theorem \ref{allgrpsthm}. 
For $\log_r (n/c_1) \geq c_2 m$ by Theorem \ref{LandSeitThm}, 
so that: 
\begin{center}
$\log_2 (n/c_1) \log_r (n/c_1)/ c_2 ^2 
\geq \log_r (n/c_1)^2 / c_2 ^2 
\geq m^2$
\end{center}
(since $2 \leq r$) and: 
\begin{center}
$(n/c_1)^{\log_2 (n/c_1)/ c_2 ^2} \geq r^{m^2} \geq \lvert S \rvert$
\end{center}
as required. 

Finally, suppose $S = Y(p^{\mu}) \in \Lie (p)$. 
The representations of almost simple groups of Lie type 
in defining charactistic were studied by Liebeck \cite{Liebeck}, 
building on work of Donkin \cite{Donkin},  
as a key step in bounding the orders of non-geometric 
maximal subgroups in classical groups. 
Recall that $n(Y,p)$ is the minimal dimension of a faithful 
irreducible projective $\overline{\mathbb{F}_p} S$-module 
(consult Subsection \ref{repsubsubsect} for the value of $n(Y,p)$ 
for each possible $S$). 

\begin{thm}[\cite{Donkin}, \cite{Liebeck} Theorem 2.1-2.3] \label{LiebThm1}
Let $S = Y(p^{\mu})$ be a finite simple group of Lie type $Y$. 
Let $M$ be a faithful absolutely irreducible projective $S$-module 
over $\mathbb{F}_{p^{\lambda}}$. Then: 
\begin{center}
$\dim (M) \geq n(Y,p) ^{\mu/(\lambda,\mu)}$. 
\end{center}
\end{thm}

Recall that $G = X(q)$ is a classical group in characteristic $p$. 
Let $n(X,p)$ be as in Corollary \ref{mainthm}. 
Note that complete lists of maximal subgroups 
of $G$ are known for $n(X,p) \leq 12$ and are recorded in 
\cite{BrHoRD} Tables 8.1-85. 
The possibilities for $S$ being read off from these tables, 
it may be verified that there is a law of the required length 
satisfied by all of them, using our induction hypothesis 
and Lemma \ref{UnionLemma}. 
We therefore henceforth assume that $n(X,p) \geq 13$. 
Write $p^{\lambda} = q^2$ if $X = {^2}A_l$ 
and $p^{\lambda} = q$ otherwise. 

First suppose that $\mu/(\lambda,\mu) = t \geq 2$. 
Then $\mu \leq \lambda t$ and (by Theorem \ref{LiebThm1}) 
$n(X,p) \geq n(Y,p) ^t$. 
By induction, $S$ satisfies a law of length: 
\begin{center}
$O_Y (p^{\mu \lfloor n(Y,p)/2 \rfloor} \log (q)^{O_Y(1)})$
\end{center} 
(since $a(Y,p) \leq \lfloor n(Y,p)/2 \rfloor$) 
which is acceptable since: 
\begin{center}
$\mu \lfloor n(Y,p)/2 \rfloor 
\leq \lambda t \lfloor n(X,p)^{1/t}/2 \rfloor 
\leq a(X,p) \lambda$
\end{center}
for $X \neq {^2}A_{l}$ and: 
\begin{center}
$\lambda t \lfloor n(X,p)^{1/t}/2 \rfloor \leq a(X,p) \lambda/2$
\end{center}
for $X = {^2}A_{l}$, since $n(X,p) \geq 13$. 
Moreover $n(Y,p) \geq 2$, 
so $t \leq \log_2 (n(X,p))$ and $\lambda \leq 2 \log_p (q)$, 
and the total number of possibilities for $S$ is at most: 
\begin{center}
$O \big(\log_p (q) \sum_{t=2} ^{\log_2 (n(X,p))} t n(X,p)^{1/t} \big)$
\end{center}
so by Lemma \ref{UnionLemma} we have a law of the required 
length satisfies by all such $S$
(the factor of $\log_p (q)^2$ introduced when we apply 
Lemma \ref{UnionLemma} is acceptable, 
since $n(Y,p) \leq \sqrt{n(X,p)}$, so that $Y \neq X$). 

Therefore we may suppose $\mu/(\lambda,\mu) =1$, so $\mu\mid\lambda$. 
Write $\lambda=\mu r$ for $r \in \mathbb{N}$. 
First, for $r \geq 12$ we may argue as in the case of 
subfield levels in geometric subgroups: 
by Proposition \ref{maxorderlawpropn} 
there is a word of length $O(q^{a(X,p)})$ 
which is a law for all possible $S$. 

Next suppose $2 \leq r \leq 11$. 
Then by induction $S$ satisfies 
a law of length: 
\begin{center}
$O_Y (p^{\mu \lfloor n(Y,p)/2 \rfloor} \log (q)^{O_Y(1)})$
\end{center} 
which is acceptable since 
$\mu \lfloor n(Y,p)/2 \rfloor\leq\lambda\lfloor n(X,p)/2 \rfloor /2$
and $\lambda \lfloor n(X,p)/2 \rfloor /2 \leq a(X,p)\lambda$
for $X \neq {^2}A_{l}$ while
$\lambda\lfloor n(X,p)/2 \rfloor /2 = a(X,p)\lambda/2$ 
for $X = {^2}A_l$. 
Moreover at most $O(n(X,p))$ groups $S$ occur in this case, 
so by Lemma \ref{UnionLemma} 
we have an acceptable law for all of them. 

Finally suppose $r=1$, so that $\lambda=\mu$. 
Once again, recall that $\tilde{S}$ is a central extension 
of $Y(q)$ (respectively $Y(q^2)$ if $X = {^2}A_l$). 
Since $n(Y,p) \leq n(X,p)$, at most boundedly many 
$Y$ arise for each $X$. 
It remains to prove that $Y(q) \preceq X(q)$ 
(respectively $Y(q^2) \preceq X(q)$)
and $a(Y,p) \leq a(X,p)$ (respectively $2a(Y,p) \leq a(X,p)$), 
from which it follows by Lemma \ref{UnionLemma} there is a sufficiently short law satisfied by all $S$ which occur. 

We may easily exclude the case of $Y$ exceptional. 
Indeed $X$ is classical and $n(X,p) \geq 13$, 
so if $Y$ is exceptional with $n(Y,p) \leq n(X,p)$, 
then $Y(q) \preceq X(q)$ 
(respectively $Y(q^2) \preceq X(q)$) 
by Definition \ref{posetdefn} (ii), 
and it may be seen by inspection of Table \ref{table:lawstable} 
above that $2 a(Y,p) \leq a(X,p)$. 

Therefore suppose $Y$ to be classical. 
We recall some more information from \cite{Liebeck} 
on the possible dimensions of irreducible 
modular representations of $\tilde{S}$. 
Let $K$ be the splitting field of $\tilde{S}$ 
(so that $K = \mathbb{F}_{p^{2\mu}}$ in the case 
$Y = {^2}A_l$ or ${^2}D_l$ and 
$K = \mathbb{F}_{p^{\mu}}$ otherwise). 

\begin{thm}[\cite{Liebeck} Theorem 1.1] \label{LiebModDimBounds}
Let $M$ be an irreducible $K \tilde{S}$-module. 
Then either (i) $\dim_K (M) = n(Y,p)$ 
(ii) $\dim_K (M) \geq n(Y,p)^2 / 2$ or (iii) 
one of the following holds: 
\begin{itemize}
\item[(a)] $\dim_K (M)$ is as in the following table: 

\begin{center}
\small
\begin{tabular}{|c|c|}
\hline 
$X$ & $\dim_K (M)$ \\ 
\hline 
$A_l$ or $^2A_l$ & $l(l+1)/2$ \\ 
\hline 
$B_l$ & $l(2l+1)$ \\ 
\hline 
$C_l$ & $\begin{array}{cc} 
l(2l-1)-2 & \text{if }p \mid l \\ 
l(2l-1)-1 & \text{if }p \nmid l \end{array}$ \\ 
\hline 
$D_l$ or $^2D_l$ & $\begin{array}{cc} 
l(2l-1) & \text{ for $p$ odd} \\ 
l(2l-1)-1 & \text{ for $p=2$, $l$ odd} \\ 
l(2l-1)-2 & \text{ for $p=2$, $l$ even}\end{array}$ \\ 
\hline 
\end{tabular} 
\end{center}

\item[(b)] $Y = B_l$ or $C_l$, $2 \leq l \leq 6$ 
and $\dim_K (M) = 2^l$; 
\item[(c)] $Y = D_l$ or $^2D_l$, $4 \leq l \leq 7$ 
and $\dim_K (M) = 2^{l-1}$; 
\item[(d)] $Y = C_3$ and $\dim_K (M) = 14$. 
\end{itemize}
\end{thm}

Suppose first that $n(Y,p) < n(X,p)$ 
(so that Definition \ref{posetdefn} (i) applies). 
Since the action of $\tilde{S}$ on the natural module $V$ 
for $G$ is absolutely irreducible, 
$M = V \otimes_{\mathbb{F}_{p^{\lambda}}} K$ is an irreducible 
$K \tilde{S}$-module satisfying $\dim_K (M) > n(Y,p)$. 
Thus $n(X,p) = \dim_K (M)$ is as in Theorem \ref{LiebModDimBounds} 
(ii) or (iii). 
We can now verify by brute computation 
(and using Theorem \ref{minrepdegreethm}) that $2 a(Y,p) \leq a(X,p)$ 
for $X$ of type $^2A_l$ and $a(Y,p) \leq a(X,p)$ otherwise. 
These computations are greatly expedited by the observation that 
$a(Y,p) \leq \lfloor n(Y,p)/2 \rfloor$, 
$a(X,p) \geq \lfloor n(X,p)/2 \rfloor-2$ (since $X$ is classical) 
and $a({^2A_l},p) = \lfloor (l+1)/2 \rfloor 
= \lfloor n({^2}A_l,p)/2 \rfloor$. 

We may therefore assume $n(X,p)=n(Y,p)$. 
There are elementary methods which could exclude 
most possibilities for $Y$ which would be potential 
obstructions to the conclusion of Theorem \ref{strongmainthm}. 
However we have found that the most efficient way of 
excluding all such $Y$ simultaneously is to use the 
conclusion of Liebeck's investigations as a black box. 

\begin{thm}[\cite{Liebeck} Theorem 4.1] \label{LiebeckOrderBd}
Let $G = X(q)$ for $X$ of classical type with natural module 
of dimension $n$. Let $H$ be a maximal subgroup of $G$. 
Then one of the following holds. 
\begin{itemize}
\item[(i)] $H$ lies in $\mathcal{C}_1 - \mathcal{C}_8$ 
(as described in Theorem \ref{structurethmclassical}) 
and $X \neq D_4$; 
\item[(ii)] $H \in \lbrace \Alt(n+1),\Sym(n+1),\Alt(n+2),\Sym(n+2) \rbrace$;
\item[(iii)] $\lvert H \rvert < q^{3 n}$ 
(respectively $\lvert H \rvert < q^{6 n}$ for $X = {^2} A_l$). 
\end{itemize}
\end{thm}

Cases (i) and (ii) of Theorem \ref{LiebeckOrderBd} 
having been dealt with above, 
we may assume that the order of $S$ satisfies the 
upper bound from Theorem \ref{LiebeckOrderBd} (iii). 
This will contradict 
the following lower bound on the orders of classical groups. 

\begin{thm} \label{ClassicalGroupOrderThm}
Let $Y$ be of classical type. Then: 
\begin{center}
$\lvert Y (q)\rvert = \Omega_Y (q^{(n(Y,p)^2-n(Y,p))/2})$. 
\end{center}
\end{thm}

\begin{proof}
The orders of the finite simple groups of Lie type 
are computed in many places, for instance \cite{Carter,Wilson}. 
\end{proof}

Since $n(X,p)=n(Y,p)\geq 13$, $3 n(X,p) < (n(Y,p)^2-n(Y,p))/2$. 
Combining Theorem \ref{LiebeckOrderBd} (iii) with 
Theorem \ref{ClassicalGroupOrderThm}, 
$\Omega_Y (q^{(n(Y,p)^2-n(Y,p))/2}) =\lvert Y(q)\rvert < q^{3n(X,p)}$ 
(respectively 
$\Omega_Y (q^{n(Y,p)^2-n(Y,p)}) =\lvert Y(q^2)\rvert < q^{6n(X,p)}$ 
for $X={^2}A_l$). This is a contradiction 
for $q$ sufficiently large depending on $n(X,p)$. 
This concludes the last case of the proof of Proposition \ref{nongeomprop}. 

\begin{rmrk}
For the sake of Remark \ref{withbab}, 
the dependence of the implied constant in 
Theorem \ref{ClassicalGroupOrderThm} may be made explicit. 
\end{rmrk}

It is important to note that the work of subection \ref{GenSect} 
and Section \ref{NonGenSect} up to this point do 
not by themselves 
constitute a proof of the upper bound in Theorem \ref{strongmainthm} 
for the classical groups. This is because 
a classical group may contain a large exceptional group 
lying prior to it in our induction, 
so that we must assume Theorem \ref{strongmainthm} for the 
exceptional subgroup in order to proceed. 
This means that our inductive argument in the proof of 
the upper bound in
Theorem \ref{strongmainthm} must handle exceptional and classical 
groups simultaneously. 

\subsection{Exceptional Groups}

The broad shape of maximal subgroups of exceptional groups of Lie type 
is very similar to that for the classical groups, 
though the Aschbacher classes $\mathcal{C}_1$-$\mathcal{C}_8$ 
are not formally defined for the exceptional groups. 
Once again, every maximal subgroup has a short subnormal 
series, all the factors of which are either 
of small order; abelian, or a group of Lie type 
for which we may assume a sufficiently short law 
exists by induction. 

\begin{propn} \label{excthm}
Let $G = X (q)$, $a(X,p)$ be as in Theorem \ref{strongmainthm}, 
with \linebreak$X \in 
\lbrace E_6,E_7,E_8,F_4,G_2,{^3}D_4,{^2}E_6,{^2}B_2,{^2}F_4,{^2}G_2\rbrace$. 
Suppose that the conclusion of Theorem \ref{strongmainthm} 
holds for all groups $H$ satisfying $H \prec G$ 
as in Definition \ref{posetdefn}. 
Then there is word in $F_2$ of length $O(q^{a(X,p)} \log (q)^{O(1)})$ 
which is a law for all maximal subgroups $M$ of $G$. 
\end{propn}

\begin{proof}
As for the classical groups, it will suffice by Lemma \ref{UnionLemma} 
to divide the maximal subgroups of $G$ into a number of classes independent 
of $q$ and to show that a law of the required length holds in each family. 

First, we deal with the class of \emph{subfield subgroups}, 
that is, subgroups of the form $M = X(q^{1/r})$, 
where $r$ is a prime divisor of $\log_p (q)$. 
A law valid in all such subgroups is produced 
by the same argument that was used for the subfield levels 
in Proposition \ref{geompropn}: 
there is a law of acceptable length valid in all such 
$M$ with $r \geq 7$ by Proposition \ref{maxorderlawpropn}. 
For each of $r = 2, 3$ or $5$, 
a law of acceptable length holds in $M$ by induction. 
Thus by Lemma \ref{UnionLemma} there is a law 
of acceptable length valid in all subfield subgroups. 

The case of the Suzuki groups $^2 B_2 (q)$ ($q = 2^{2m+1}$) 
follows from \cite{BrHoRD} Table 8.16: 
every maximal subgroup of $G$ is either a \emph{subfield subgroup} or is soluble of derived length at most $3$. 

The case of the small Ree groups $^2 G_2 (q)$ ($q = 3^{2m+1}$)
follows from \cite{BrHoRD} Table 8.43 (which in turn is based on \cite{KleidG2}): 
every maximal subgroup is either 
a \emph{subfield subgroup}; 
is soluble of derived length at most $4$, 
or is isomorphic to $2 \times \PSL_2 (q)$. 

The case of the Steinberg triality groups $^3 D_4 (q)$ follows from \cite{BrHoRD} Table 8.51 (which in turn is based on \cite{Kleid3D4}). 
Other than the \emph{subfield subgroups}, 
all maximal subgroups $M$ of $G$ 
have the following structure: 
there is a subnormal series 
\begin{center}
$M = K_1 \vartriangleright K_2 \vartriangleright K_3 
\vartriangleright K_4 \vartriangleright K_5 = 1$
\end{center}
such that $K_1/K_2$ has order at most $24$; 
$K_2/K_3$ is abelian; $K_4$ has order a power of $q$, 
and $K_3/K_4$ is isomorphic to one of the following: 
\begin{center}
$\SL_2 (q^3) \circ (q-1)$; $\SL_2 (q) \circ (q^3-1)$; 
$G_2 (q)$; $\SL_2 (q^3) \circ \SL_2 (q)$; \\
$\SL_3 (q) \circ (q^2+q+1)$; $\SU_3 (q) \circ (q^2-q+1)$; 
$\PGL_3(q)$; $\PGU_3(q)$. 
\end{center}
There is a sufficiently short law for $K_4$ by Corollary \ref{uniplawcoroll}, 
for $K_3 / K_4$ by hypothesis and Corollary \ref{ProdLemma}, 
and thus also for $M$ by Lemma \ref{ExtnLemma}. 

The case of the large Ree groups $^2 F_4 (q)$ ($q = 2^{2m+1}$) 
follows from the Main Theorem of \cite{Malle} 
(note that the order $q$ of the underlying field is, for historical reasons, 
denoted $q^2$ in \cite{Malle}, in spite of being an odd power of $2$). 
Several isomorphism types of groups which are extensions of an abelian group 
by a group of bounded order occur, as do \emph{subfield subgroups}. 
Other than these, every maximal subgroup $M$ has a subnormal series: 
\begin{center}
$M = K_1 \vartriangleright K_2 \vartriangleright K_3 
\vartriangleright K_4 = 1$
\end{center}
such that $K_1 / K_2$ is abelian, $K_3$ is a $2$-group, and $K_2 / K_3$ 
is isomorphic to one of the following: 
\begin{center}
$\PSL_2 (q) \times (q-1)$; $^2B_2(q) \times (q-1)$; \\
$\SU_3 (q)$; $\PGU_3 (q)$; $^2B_2(q) \times {^2B_2(q)}$; 
$B_2 (q)$. 
\end{center}
We have a sufficiently short law for $K_3$ by Corollary \ref{uniplawcoroll}, 
for $K_2 / K_3$ by hypothesis and Corollary \ref{ProdLemma}, 
and thus also for $M$ by Lemma \ref{ExtnLemma}. 

We may therefore assume that $G$ is one of $G_2 (q)$, $F_4 (q)$, $E_6 (q)$, 
$^2 E_6 (q)$, $E_7 (q)$ or $E_8 (q)$. 
Theorem 8 of \cite{LieSeiSurv} describes the possibilities for $M$: 
\begin{itemize}
\item[{\it (i-a)}] $M$ is a maximal parabolic subgroup; 

\item[{\it (i-b)}] $M$ is a subgroup of \emph{maximal rank}: 
the possibilities for $M$ are given 
in Tables 5.1 and 5.2 from \cite{LieSaxSei}; 

\item[{\it (i-c)}] $G = E_7 (q)$ and $M = (2^2 \times P\Omega_8 ^+ (q).2^2).\Sym(3)$ or $^3 D_4 (q).3$; 

\item[{\it (i-d)}] $G = E_8 (q)$ and $M = \PGL_2 (q) \times \Sym(5)$; 

\item[{\it (i-e)}] $F^* (M)$ is one of the groups appearing in Table 3 from \cite{LieSeiSurv}; 

\item[{\it (ii)}] $M$ is either a \emph{subfield subgroup} 
or a \emph{twisted type}, that is one of the following: 
$^2E_6(q^{1/2}) \leq E_6(q)$ (for $q$ a square); 
$^2 F_4 (q) \leq F_4 (q)$ (for $q = 2^{2m+1}$) 
or $^2G_2(q) \leq G_2(q)$ (for $q = 3^{2m+1}$); 

\item[{\it (iii)}] $M$ is an \emph{exotic local subgroup}, isomorphic to one of 
$2^3 . \SL_3 (2)$, $3^3 . \SL_3 (3)$, $3^{3+3} . \SL_3 (3)$, 
$5^3 . \SL_3 (5)$ or $2^{5+10} . \SL_5 (2)$; 

\item[{\it (iv)}] $G = E_8 (q)$ and $M = (\Alt(5) \times \Alt(6)).2^2$; 

\item[{\it (v)}] $F^* (M)$ is one of the simple groups appearing in Table 2 from \cite{LieSeiSurv}; 

\item[{\it (vi)}] $F^* (M)$ is a finite simple group of Lie type over a field $\mathbb{F}_{q_0}$ of characteristic $p$; $\rk (F^* (M)) \leq \rk(G)/2$, and there exists a constant $t(G)$, depending only on the root system of $G$, 
such that one of the following holds:
\begin{itemize}
\item[{\it (a)}] $q_0 \leq 9$;
\item[{\it (b)}] $F^* (M) = \PSL_3 (16)$ or $\PSU_3 (16)$;
\item[{\it (c)}] $q_0 \leq t(G)$ and $F^* (M) = \PSU_2 (q_0)$, $^2 B_2 (q_0)$ or $^2 G_2 (q_0)$. 
\end{itemize}

\end{itemize}
Recall that $F^*(M)$ is the generalized Fitting subgroup of $M$. 

In case (ia), there exists a subnormal series: 
\begin{center}
$M = K_1 \vartriangleright K_2 \vartriangleright K_3 \vartriangleright K_4 
\vartriangleright K_5 = 1$
\end{center}
such that $K_1/K_2$ is of bounded order, $K_2/K_3$ is abelian, 
$K_4$ is a $p$-group, and $K_3/K_4 = L \circ A$, 
where $A$ is abelian and $L$ is a 
central extension of a group $T \leq \overline{L} \leq \Aut(T)$, 
for $T$ one of the following direct products of nonabelian 
simple groups. 

\begin{table}[H] \label{lawstable}
\small
\begin{center}
\begin{tabular}{|c|c|}
\hline 
$G$ & $T$ \\ 
\hline 
$G_2 (q)$ & $A_1 (q)$ \\ 
\hline 
$F_4 (q)$ & $B_3 (q)$; $C_3 (q)$; 
$A_1 (q) \times A_2 (q)$ \\ 
\hline 
$E_6 (q)$ & $D_5(q)$; $A_1 (q) \times A_4 (q)$; \\
&$A_1 (q) \times A_2 (q) \times A_2 (q)$; $A_5(q)$\\ 
\hline 
${^2}E_6 (q)$ & $^2A_5(q)$; $A_1 (q) \times A_2 (q^2)$; 
$A_2(q) \times A_1 (q^2)$; $^2D_4(q)$ \\ 
\hline
$E_7 (q)$ & $E_6 (q)$; $A_1 (q) \times D_5(q)$; 
$A_2 (q) \times A_4 (q)$; $D_6 (q)$; \\ 
& $A_1 (q) \times A_5 (q)$; 
$A_1 (q) \times A_2 (q) \times A_3 (q)$; $A_6(q)$ \\
\hline
$E_8 (q)$ & $E_7(q)$; $A_1 (q) \times E_6 (q)$; 
$D_5 (q)\times A_2 (q)$; $A_3 (q)\times A_4 (q)$; \\ 
 & $A_1 (q) \times A_2 (q) \times A_4 (q)$; 
$D_7 (q)$; $A_1 (q)\times A_6 (q)$; $A_7 (q)$\\
\hline 
\end{tabular} 
\end{center}
\end{table}

Each simple factor of each $T$ arising satisfies 
a sufficiently short law by induction, 
and by the Schreier hypothesis $\Out(T)$ is the extension 
of a soluble group of derived length at most three, 
by a permutation group of the isomorphic direct factors. 
By Lemmas \ref{ProdLemma} and \ref{ExtnLemma}, therefore, 
$L$ satisfies a sufficiently short law. 

In case (ib), for every group $M$ 
appearing in Table 5.1 from \cite{LieSaxSei}, 
there exists a subnormal series: 
\begin{center}
$M = K_1 \vartriangleright K_2 \vartriangleright K_3 \vartriangleright K_4 
\vartriangleright K_5 = 1$
\end{center}
such that $K_1/K_2$ is of bounded order, $K_2/K_3$ and $K_4$ are abelian, 
and $K_3/K_4$ is a direct product of a bounded number of 
finite simple groups $H_i \in \Lie (p)$, 
each satisfying $H_i \prec G$. 
Meanwhile every group appearing in Table 5.2 from \cite{LieSaxSei} 
is the extension of an abelian group by a group of bounded order. 

In cases (ic) and (id), we have sufficiently short laws for 
$P\Omega_8 ^+ (q)$, $^3D_4(q)$ and $\PGL_2(q)$ by hypothesis, 
so we may produce sufficiently short laws for $M$ by Lemma \ref{ExtnLemma}. 

For case (ie), $F^{\ast}(M)$ is the 
direct product of at most $3$ nonabelian finite 
simple groups of Lie type in characteristic $p$, 
each of which satisfies a sufficiently short law by induction 
(seen by inspection of Table 3 from \cite{LieSeiSurv}). 
By Proposition \ref{fittingprop}, $M \leq \Aut(F^{*}(M))$. 
By Theorem \ref{simpleautthm} and Proposition \ref{prodautprop}, 
$\Out(F^{*}(M))$ is the extension of a soluble group of derived 
length at most $3$ by a subgroup of $\Sym(3)$. 
By Lemma \ref{ExtnLemma}, it suffices 
that $F^{\ast}(M)$ satisfies a sufficiently short law. 
This is so, since each simple factor does, 
and by Corollary \ref{ProdLemma}. 

Among the subgroups arising in case (ii), 
the untwisted subfield subgroups 
(that is, those of the form $M = X(q^{1/r})$ 
for $r$ a prime divisor of $\log_p (q)$) 
are dealt with as above. 
Those of twisted type 
satisfy a sufficiently short law by induction. 

Finally the subgroups arising in cases (iii)-(vi) are all of bounded order (in cases (v) and (vi) this follows from Proposition 
\ref{fittingprop}, since $F^*(M)$ is simple of bounded order). 
\end{proof}

\section{Completing the Proof of Theorem \ref{strongmainthm}} 
\label{MainThmProofSect}

At last we are ready to put everything together 
and prove our main result. 

\begin{proof}[Proof of Theorem \ref{strongmainthm}]
\normalfont
First we show that a law of the required length 
does indeed exist. This is the upshot of the last two sections. 
To be explicit: let $G$ be as in Theorem \ref{strongmainthm}. 
Let $w_{\rm gen} \in F_2$ be the word produced in Theorem 
\ref{gensetmainthm}. 
Suppose by induction that a law of the required length 
exists for all $H \prec G$ as in Definition \ref{posetdefn}. 
If $G$ is classical, let $w_{\rm geom}$ and $w_{\rm nongeom}$ 
be as in Propositions \ref{geompropn} and \ref{nongeomprop}, 
respectively. 
If $G$ is exceptional, let $w_{\rm exc}$ be as 
in Proposition \ref{excthm}. 
For $(g,h) \in G \times G$, either $\lbrace g,h\rbrace$ 
generates $G$ or $\langle g,h \rangle$ is contained 
in a maximal subgroup of $G$. 
If $G$ is classical then this subgroup is either 
geometric or non-geometric. 
Thus applying Lemma \ref{UnionLemma} to 
$w_{\rm gen},w_{\rm geom},w_{\rm nongeom}$ (for $G$ classical) 
or $w_{\rm gen},w_{\rm exc}$ (for $G$ exceptional) we have the 
required law. 

\vspace{0.1cm}

The lower bound on the length of the shortest law in 
$G$ from Theorem \ref{strongmainthm} is witnessed by the following subgroups, which exist for $q$ 
larger than an absolute constant. There is no other way than doing this case by case.

\begin{itemize}
\item By Theorem \ref{lowerboundex}, 
$A_l (q)$ has shortest law of length 
$\Omega (q^{\lfloor (l+1)/2\rfloor})$. 
\item $^2 A_1 (q)$ has shortest law of length $\Omega(q)$ 
by Theorem \ref{smallrkisothm} (i). 
Suppose $l \geq 2$ and let $n=l+1$. 
Claim that the shortest law for $^2 A_l (q)$ 
has length $\Omega (q^{\lfloor n/2\rfloor})$. 
If $n$ is odd, then $^2 A_{l-1} (q)$ is 
a subquotient of $^2 A_l (q)$ (\cite{BrHoRD} Table 2.3) 
and the result follows by induction. 
If $n$ is divisible by $4$, then $A_{n/2-1}(q^2)$ 
is a subquotient of $^2 A_l (q)$ (\cite{BrHoRD} Table 2.3 again) 
and the conclusion follows. 
Otherwise write $n=ms$ for $s$ an odd prime. 
Then $^2 A_{m-1} (q^s)$ is a subquotient 
of $^2 A_l (q)$ (\cite{BrHoRD} Table 2.6). Since $m$ is even, 
we are done by induction. 
\item $C_1 (q)$ has shortest law of length $\Omega(q)$ 
by Theorem \ref{smallrkisothm} (i). 
For $l \geq 2$, claim the shortest law for $C_l (q)$ 
has length $\Omega (q^l)$. 
Write $l=ms$, for $s$ a prime. 
Then $C_m (q^s)$ is a subquotient of $C_l (q)$ 
(\cite{BrHoRD} Table 2.6), so we are done by induction. 

\item By Theorem \ref{smallrkisothm} (iv) and (vii) 
and the previous paragraphs, $^2 D_2 (q)$ 
and $^2 D_3 (q)$ each have shortest law of length $\Omega(q^2)$. 
Assume $l \geq 4$ and claim the shortest law for $^2 D_l (q)$ 
has length $\Omega(q^{2\lfloor l/2\rfloor})$. 
If $l$ is even, and $l/2 = ms$, for $s$ a prime, 
then $^2 D_{l/s} (q^s)$ is a subquotient 
of $^2 D_l (q)$ (\cite{BrHoRD} Table 2.6). 
If $l$ is odd, then $^2 D_{l-1} (q)$ is a subquotient 
of $^2 D_l (q)$ (\cite{BrHoRD} Table 2.3). 
In both cases we are done by induction. 

\item By Theorem \ref{smallrkisothm} (v) and the above, 
$B_2 (q)$ has shortest law of length $\Omega (q^2)$. 
Suppose $l \geq 3$. 
By Theorem \ref{smallrkisothm} (viii) 
we need only consider $B_l(q)$ for $q$ odd. 
If $l$ is even, 
then $^2 D_l (q)$ is a subquotient of $B_l(q)$, 
whereas if $l$ is odd, 
then $^2 D_{l-1} (q)$ is a subquotient of $B_l(q)$ 
(\cite{BrHoRD} Table 2.3). 
In both cases it follows from the previous paragraph 
that the shortest law for 
$B_l(q)$ has length $\Omega (q^{2\lfloor l/2\rfloor})$. 

\item For $D_l (q)$, with $l \geq 4$, 
then $^2 D_{l-1} (q)$ is a subquotient of $D_l (q)$. 
If $l$ is odd, it follows that the shortest law for $D_l (q)$ 
has length $\Omega(q^{l-1})$, whereas if $l$ is even, 
it has length $\Omega(q^{l-2})$. 
Moreover if $l$ and $q$ are both even, 
then $C_{l-1} (q)$ is a subquotient of $D_l (q)$, 
whose shortest law therefore has length $\Omega(q^{l-1})$ 
(see \cite{BrHoRD} Table 2.3). 

\item $G_2 (q)$ has a parabolic subgroup with Levi factor $\GL_2(q)$: 
this has shortest law of length $\Omega( q)$. 

\item $^2 G_2 (q)$ has a subgroup $\PSL_2 (q)$: 
this has shortest law of length $\Omega( q)$; 
as is noted in \cite{KleidG2} Theorem C, 
this subgroup is contained in an involution-centraliser, 
and exists for $q \geq 27$. 

\item $^3 D_4 (q)$ has a maximal subgroup with a subquotient 
$\PSL_2 (q^3)$: this has shortest law of length $\Omega(q^3)$; 
as noted in \cite{Kleid3D4}, 
this is a maximal parabolic subgroup. 

\item $F_4 (q)$ has a maximal subgroup with $\Omega_9(q)$ 
as a quotient: this has shortest law of length $\Omega(q^4)$; 
this is a subgroup of maximal rank 
(see Table 5.1 from \cite{LieSaxSei}). 
The subgroup $\Sp_4 (q^2)$, occuring inside a 
maximal subgroup of maximal rank as noted above,
also has shortest law of length $\Omega(q^4)$, 
but as is noted in \cite{LieSaxSei}, 
this only arises for $q$ even. 

\item $^2 F_4 (q)$ has a subquotient $B_2 (q)$, 
whose shortest law has length $\Omega(q^2)$ 
(see the Main Theorem of \cite{Malle}; 
note however the difference in notational convention: 
our $q$ is denoted $q^2$ in \cite{Malle}). 

\item $E_6 (q)$ has a parabolic subgroup, the Levi factor of which 
contains $\Omega_{10} ^+ (q)$, 
whose shortest law has length $\Omega (q^4)$. 

\item $^2E_6(q)$ has a parabolic subgroup, the Levi factor of which 
contains $\Omega_{8} ^- (q)$, 
whose shortest law has length $\Omega (q^4)$. 

\item $E_7(q)$ has a maximal subgroup of maximal rank containing 
$\PSL_2(q^7)$ (see Table 5.1 from \cite{LieSaxSei}), 
whose shortest law has length $\Omega (q^7)$.

\item $E_8 (q)$ has a maximal parabolic subgroup, 
the Levi factor of which contains $E_7(q)$, 
whose shortest law has length $\Omega (q^7)$.

\item Finally, the shortest law in $^2B_2(q)$ has length 
$\Omega(q^{1/2})$, by \cite{BGTSuzuki} Lemma 3.4. 
\end{itemize}
This finishes the proof of Theorem \ref{strongmainthm}.
\end{proof}

\section*{Acknowledgments}

This research was supported by ERC CoG 681207 
and ERC grant "GRANT" no. 648329. 
HB would like to thank Tim Burness, Martin Liebeck, 
Nikolay Nikolov and Emilio Pierro for enlightening 
conversations. 

\appendix

\section{Background on Groups of Lie Type} \label{theappendix}

The goal of this appendix is to provide background on finite simple groups and  prove 
Proposition \ref{orderdivpropn}. We also collect various tables and case studies in order to make the main text more transparent. We start with a quick overview that recalls the classical theory.

\subsection{Groups of Lie Type}

\subsubsection{Isomorphisms in Small Rank or Characteristic}

We note some exceptional isomorphisms of low-dimensional 
classical groups. 
Items (i)-(vii) of Theorem \ref{smallrkisothm} below are recorded, 
for example, in \cite{BrHoRD} Proposition 1.10.1. 
Item (viii) is discussed for instance 
in \cite{BrHoRD} Subsection 1.5.5.

\begin{thm} \label{smallrkisothm}
Let $q$ be a prime power. 
\begin{itemize}
\item[(i)] $\PSL_2 (q) \cong \PSp_2(q) \cong \PSU_2(q) \cong P\Omega^{\circ} _3(q)$; 

\item[(ii)] $P\Omega_2 ^{\pm} (q)$ are abelian. 

\item[(iii)] $P\Omega_4 ^+ (q) \cong \PSL_2 (q) \times \PSL_2 (q)$;

\item[(iv)] $P\Omega_4 ^- (q) \cong \PSL_2 (q^2)$; 

\item[(v)] $P\Omega^{\circ} _5 (q) \cong \PSp_4(q)$; 

\item[(vi)] $P\Omega_6 ^+ (q) \cong \PSL_4(q)$; 

\item[(vii)] $P\Omega_6 ^- (q) \cong \PSU_4(q)$; 

\item[(viii)] Suppose $q$ is even. 
Then for all $n \geq 1$, 
$P\Omega^{\circ} _{2n+1}(q) \cong \PSp_{2n}(q)$. 
\end{itemize}
\end{thm}

It follows that, in proving Theorem \ref{strongmainthm}, 
we may assume the following restrictions on $X$ and $q$ hold. 

\begin{itemize}
\item[{\it (i)}] If $X = {^2}A_l$ or $C_l$ then $l \geq 2$; 

\item[{\it (ii)}] If $X = B_l$ then $l \geq 3$; 

\item[{\it (iii)}] If $X = D_l$ or $^2 D_l$ then $l \geq 4$; 

\item[{\it (iv)}] If $X = B_l$ then $q$ is odd. 

\end{itemize}

\subsubsection{Representations} \label{repsubsubsect}

We record the minimal dimension $n = n(X,p)$ of a faithful projective 
representations of the simple group $G=X(q)$ 
over the algebraic closure of $\mathbb{F}_q$ 
(recalling that $q$ is a power of the prime $p$: 
implicit in writing $n=n(X,p)$ is the claim that 
the dimension does not depend on the degree of the 
field extension $(\mathbb{F}_q \colon \mathbb{F}_p)$; 
we see below that this is indeed the case). 
In addition to providing context to the statement of 
Corollary \ref{mainthm}, 
knowing $n(X,p)$ will be important in the proof of 
Theorem \ref{strongmainthm}, in that it will provide a 
restriction on the possible embeddings of one 
quasisimple group of Lie type into another of matched 
characteristic as a non-geometric subgroup. 

The various possible values for $n(X,p)$ 
are recorded in \cite[Proposition 5.4.13]{KleLie}. 
We reproduce their conclusions in the next Theorem. 
Note that every group in our list 
is isomorphic to a group in their list 
by Theorem \ref{smallrkisothm}. 

\begin{thm} \label{minrepdegreethm}
Let $G = X (q)$ be as in Theorem \ref{strongmainthm}. 
If $X=C_2$ then suppose $q \geq 3$. 
Let $n=n(X,p) \in \mathbb{N}$ be 
the minimal dimension of a faithful projective module 
for $G$ over $\overline{\mathbb{F}}_q$ 
(in other words, let $n$ be minimal such that $G$ is isomorphic to a subgroup 
of $\PGL_n ({\overline{\mathbb{F}}_q})$). 
Then $n$ is as in Table \ref{reptable} below. 
\end{thm}

\begin{table}[H] 
\small
\begin{center}
\begin{tabular}{|c|c|c|c|c|c|}
\hline 
$X$ & $A_l \text{ or } ^2 A_l$
 & $B_l \text{ ($p$ odd)}$ & $C_l$ & $D_l$ & $^2 D_l$ \\ 
\hline 
$n$ & $l+1$ & $\begin{array}{cc} 2 & (l=1) \\ 4 & (l=2) \\ 2l+1 & (l \geq 3) \end{array}$ & $2l$ & $\begin{array}{cc} 4 & (l=2) \\ 4 & (l=3) \\ 2l & (l \geq 4) \end{array}$ & $\begin{array}{cc} 2 & (l=2) \\ 4 & (l=3) \\ 2l & (l \geq 4) \end{array}$ \\

\hline 
\end{tabular} 
\end{center}
\begin{center}
\begin{tabular}{|c|c|c|c|c|c|}
\hline 
$X$ & $E_6 \text{ or } ^2 E_6$ & $E_7$ & $E_8$ & $F_4$ & $G_2$ \\ 
\hline 
$n$ & $27$ & $56$ & $248$ & $\begin{array}{cc} 25 & \text{($p=3$)} \\ 26 & \text{($p \neq 3$)} \end{array}$ & 
$\begin{array}{cc} 7 & \text{($p$ odd)} \\ 6 & \text{($p=2$)} \end{array}$ \\
\hline 
\end{tabular} 
\end{center}
\begin{center}
\begin{tabular}{|c|c|c|c|c|}
\hline 
$X$ & $^3D_4$ & $^2B_2$ & $^2F_4$ & $^2G_2$ \\ 
\hline 
$n$ & $8$ & $4$ & $26$ & $7$ \\
\hline 
\end{tabular} 

\caption{Minimal dimensions of projective representations
of simple groups of Lie type}\label{reptable}
\end{center}
\end{table}

\begin{rmrk}
Corollary \ref{mainthm} follows from Theorem \ref{strongmainthm}, 
by inspection of Tables \ref{table:lawstable} and \ref{reptable}, 
since $a(X,p) \leq \lfloor n(X,p)/2 \rfloor$ in all cases. 
\end{rmrk}

\subsubsection{Maximal Element Orders} \label{maxeltordersubsect}

We give upper bounds on the orders of elements 
in finite simple groups of Lie type. 
For this we refer to the tables from \cite{Thom}, 
which in turn are based on \cite{KanSer}. 

\begin{propn} \label{MaxEltOrderBd}
Let $d(X)$ be as in Table \ref{table:maxelordertable}. Then: 
\begin{center}
$\max \lbrace o(g) : g \in X(q) \rbrace = O(q^{d(X)})$. 
\end{center}
\end{propn}

\begin{table}[H] 
\small
\begin{center}
\begin{tabular}{|c|c|c|c|c|c|c|}
\hline 
$X$ & $A_l$ & $^2 A_l$
 & $B_l$ & $C_l$ & $D_l$ & $^2 D_l$ \\ 
\hline 
$d(X)$ & $l$ & $l$
& $l$ & $l$ & $l$ & $l$ \\
\hline 
\end{tabular} 
\end{center}
\begin{center}
\begin{tabular}{|c|c|c|c|c|c|c|c|c|c|c|}
\hline 
$X$ & $E_6$ & $^2 E_6$ & $E_7$ & $E_8$ & $F_4$ & $G_2$
& $^3D_4$ & $^2B_2$ & $^2F_4$ & $^2G_2$ \\ 
\hline 
$d(X)$ & $6$ & $6$ & $7$ & $8$ & $4$ & $2$
& $4$ & $1$ & $2$ & $1$ \\
\hline 
\end{tabular} 
\caption{Degree of length of laws coming from maximal 
element orders}\label{table:maxelordertable}
\end{center}
\end{table}

\subsection{Automorphisms of Finite Groups}

If $G$ is a nonabelian finite simple group, 
then $Z(G)=1$, so we naturally have a short exact sequence 
$1\rightarrow G \rightarrow \Aut(G)\rightarrow \Out(G)\rightarrow 1$. 
The following is widely known and classical. 

\begin{thm} \label{simpleautthm}
Let $G$ be either a finite simple group of Lie type or 
$\Alt(n)$ for $n \geq 5$. 
Then $\Out(G)$ is soluble of derived length at most $3$. 
\end{thm}

The famous Schreier conjecture extends the conclusion 
of Theorem \ref{simpleautthm} to all nonabelian finite simple 
groups, but this is known only as a consequence of CFSG, 
and we shall not need it. 

The structure of automorphism groups 
of direct products of nonabelian finite simple groups is folklore. 

\begin{propn} \label{prodautprop}
Let $G_1 , \ldots , G_m$ be nonabelian finite simple groups. 
Then there exists $H \leq \Sym(m)$ such that: 
\begin{center}
$\Out(G_1 \times \ldots \times G_m) 
\leq (\Out(G_1) \times \ldots \times \Out(G_m)).H$. 
\end{center}
\end{propn}

For a general finite group $G$, let $F^{*}(G)$ be the 
\emph{generalized Fitting subgroup}. 

\begin{propn} \label{fittingprop}
$C_G(F^*(G)) \leq F^*(G)$. 
In particular, if $F^*(G)$ is a product of nonabelian 
finite simple groups, then $G$ embeds as a subgroup 
of $\Aut(F^*(G))$. 
\end{propn}

\subsection{Algebraic Groups and Chevalley Groups} \label{AlgGrpSubsect}

We recall some concepts from the theory of algebraic groups, 
well-covered in any standard text on the subject 
such as \cite{MalTes}. 

Let $\mathbb{G}$ be a connected semisimple linear algebraic group 
of rank $l$ over an algebraically closed field $K$ 
(so that $\mathbb{G}$ is isomorphic to a closed subgroup 
of $\GL_d (K)$). 
An element of $\mathbb{G}$ is \emph{semisimple} if it is 
diagonalizable. 
A \emph{torus} in $\mathbb{G}$ is a connected subgroup consisting 
of simultaneously diagonalizable elements; 
a subgroup of $\mathbb{G}$ is a torus iff it 
is isomorphic to $(K^{\ast})^m$, for some $m \in \mathbb{N}$. 
A \emph{maximal torus} in $\mathbb{G}$ is a torus of maximal dimension; 
$\mathbb{G}$ contains at least one maximal torus, 
all maximal tori are conjugate in $\mathbb{G}$, 
and every torus is contained in at least one maximal torus. 
The rank $l$ of $\mathbb{G}$ is by definition the common dimension 
of the maximal tori in $\mathbb{G}$; hence a torus in $\mathbb{G}$ is maximal 
iff it is isomorphic to $(K^{\ast})^l$. 

Let $\mathfrak{g}$ be the Lie algebra of $\mathbb{G}$ 
and let $T$ be a maximal torus of $\mathbb{G}$. 
Let $X(T) = \Hom (T,K^{\ast})$ be the \emph{character group} of $T$. 
$X(T)$ is a group under pointwise multiplication; 
since $T \cong (K^{\ast})^l$, $X(T) \cong \mathbb{Z}^l$. 
For $\alpha \in X(T)$, let: 
\begin{center}
$\mathfrak{u}_{\alpha} = \lbrace x \in \mathfrak{g} 
: \forall t \in T, (\Ad (t))(x) = \alpha (t)x \rbrace$. 
\end{center}
Let $\Phi (T,\mathbb{G}) = \lbrace \alpha \in X(T) 
: \mathfrak{u}_{\alpha} \neq 0 \rbrace$ be the set of 
\emph{roots of $\mathbb{G}$ relative to $T$}. 
Embedding $X(T)$ into $X(T) \otimes_{\mathbb{Z}} \mathbb{R}$, 
identified with Euclidean $l$-space, 
$\Phi (T,\mathbb{G})$ is an abstract root system 
(see \cite{MalTes} \textsection 8.1, \textsection 9.1). 

For each $\alpha \in \Phi (T,\mathbb{G})$, 
there is a unique closed connected one-dimensional subgroup 
$U_{\alpha}$ of $\mathbb{G}$ with Lie algebra $\mathfrak{u}_{\alpha}$, 
normalized by $T$ (the \emph{root subgroup} corresponding to $\alpha$). 
There is an isomorphism $\lambda \mapsto x_{\alpha} (\lambda)$ 
from $(K,+)$ to $U_{\alpha}$ satisfying: 
\begin{equation} \label{RootActEqn}
t x_{\alpha} (\lambda) t^{-1} = x_{\alpha} (\alpha (t)\lambda)
\end{equation}
for all $t \in T$ (see \cite{MalTes} Theorem 8.17). 

It is clear that every semisimple $s$ element of $\mathbb{G}$ 
is contained in at least one maximal torus. 
The element $s$ is called \emph{regular} if it lies in a unique maximal torus. 

\begin{propn}[\cite{MalTes} Corollary 14.10] \label{mainregeltpropn}
For $s \in \mathbb{G}$ semisimple, and $T$ a maximal torus in 
$\mathbb{G}$ containing $s$, 
the following conditions are equivalent. 
\begin{itemize}
\item[(i)] $s$ is regular; 
\item[(ii)] For  every root $\alpha$ of $\mathbb{G}$ relative to $T$, 
$\alpha (s) \neq 1$. 
\end{itemize}
\end{propn}

\begin{propn} \label{RegTorusNormCoroll}
Let $s \in \mathbb{G}$ be regular semisimple, 
let $T$ be the unique maximal torus of $\mathbb{G}$ containing $s$ 
and let $g \in \mathbb{G}$ be such that $s^g \in T$. 
Then $g \in N_{\mathbb{G}} (T)$. 
In particular $C_{\mathbb{G}} (s) \subseteq N_{\mathbb{G}} (T)$. 
\end{propn}

\begin{proof}
$T^g$ is also a maximal torus of $\mathbb{G}$ and $s^g$ is regular semisimple. 
By uniqueness of $T^g$, $T = T^g$. 
\end{proof}

\begin{thm} \label{WeylThm}
For every maximal torus $T$ of $\mathbb{G}$, 
$N_{\mathbb{G}} (T) / T$ is isomorphic to the Weyl group 
$W_{\mathbb{G}}$ of $\mathbb{G}$. 
\end{thm}

\begin{proof}
By \cite{MalTes} Corollary 6.5, all maximal tori are conjugate in 
$\mathbb{G}$, so the isomorphism-type of $N_{\mathbb{G}} (T) / T$ 
is independent of $T$. 
\end{proof}

We recall the construction of the Chevalley groups given in \cite{Carter}
and use this to give an explicit linear representation 
of the untwisted finite simple groups of Lie type. 
Let $K = \overline{\mathbb{F}_p}$ be the algebraic closure of 
the finite field with $p$ elements, 
and let $\Phi$ be an irreducible root system 
of rank $l$, with fundamental system of roots $\Pi$. 
Let $\mathcal{L}_K$ be the $K$-Lie algebra 
defined in \cite{Carter} \textsection 4.4; 
$\mathcal{L}_K$ has a \emph{Chevalley basis} 
$\lbrace h_{\alpha} : \alpha \in \Pi \rbrace 
\cup \lbrace e_{\beta} : \beta \in \Phi \rbrace$, 
with respect to which the structure constants 
are as in \cite[Theorem 4.2.1]{Carter}. 

Let $\mathcal{L} (K) \leq \GL_{\lvert \Phi \rvert + \lvert \Pi \rvert} (K)$ 
be the group generated by the automorphisms 
$\lbrace x_{\beta}(\lambda) : \beta \in \Phi, \lambda \in K \rbrace$ 
of the Lie algebra $\mathcal{L}_K$ 
(the $x_{\beta}(\lambda)$ are defined in \cite{Carter} \textsection 4.4). 
Then $\mathcal{L}$ is a simple adjoint algebraic group of rank $l$ and 
$\mathcal{L}_K$ is its Lie algebra. 

\begin{defn} \label{UntwistedGrpDefn}
For $q$ a power of $p$, 
the (untwisted) finite simple group of type $\Phi$ over the field $\mathbb{F}_q$ 
is the subgroup $\mathcal{L} (\mathbb{F}_q)$ of 
$\mathcal{L} (K)$ 
generated by $\lbrace x_{\beta}(\lambda) 
: \beta \in \Phi, \lambda \in \mathbb{F}_q \rbrace$. 
\end{defn}

We describe an explicit maximal torus in $\mathcal{L} (K)$. 
By Proposition 6.4.1 of \cite{Carter}, 
for $\alpha \in \Pi$ there is a homomorphism 
$\phi_{\alpha} : K^{\ast} \rightarrow \mathcal{L} (K)$ given by: 
\begin{equation} \label{Chevtorusactioneqn}
\phi_{\alpha} (z) h_{\gamma} = h_{\gamma} \text{ for $\gamma \in \Pi$, }
\phi_{\alpha} (z) e_{\beta} = z^{A_{\alpha,\beta}} e_{\beta} \text{ for $\beta \in \Phi$,}
\end{equation}
where, for $\alpha,\beta \in \Phi$, $A_{\alpha,\beta}
= 2 (\alpha,\beta)/(\alpha,\alpha) \in \mathbb{Z}$ 
is the \emph{Cartan integer}. 
By Lemma 6.4.4 of \cite{Carter}, $\phi_{\alpha} (z) \in \mathcal{L}(K)$ 
may be explicitly realized as: 
\begin{equation} \label{ChevtorusWordeqn}
\phi_{\alpha} (z) = x_{\alpha}(z) x_{-\alpha}(-z^{-1}) x_{\alpha}(z) 
x_{\alpha}(-1) x_{-\alpha}(1) x_{\alpha}(-1)
\end{equation}
By (\ref{Chevtorusactioneqn}), for any $z \in K^{\ast}$, 
$\phi_{\alpha} (z)$ is a diagonal matrix and 
we have a homomorphism $\phi : (K^{\ast})^{\Pi} \rightarrow \mathcal{L} (K)$ 
defined, for $\underline{z} = (z_{\alpha})_{\alpha \in \Pi} \in (K^{\ast})^{\Pi}$, 
by: 
\begin{equation*} 
\phi (\underline{z}) = \prod_{\alpha \in \Pi} \phi_{\alpha} (z_{\alpha})
\end{equation*}
so that: 
\begin{equation*}
\phi (\underline{z}) h_{\gamma} = h_{\gamma} \text{ for $\gamma \in \Pi$, }
\phi (\underline{z}) e_{\beta} = \phi (\underline{z})_{\beta} e_{\beta}, 
\end{equation*}
where we write: 
\begin{equation} \label{charactereqn}
\phi (\underline{z})_{\beta} = \prod_{\alpha \in \Pi} z_{\alpha} ^{A_{\alpha,\beta}}
\end{equation}
for $\beta \in \Phi$. 
From this description, it is clear that $\phi$ 
is a morphism of algebraic groups. 

\begin{lem} \label{morphfinkerlem}
The morphism $\phi$ has finite kernel. 
\end{lem}

\begin{proof}
The Cartan matrix $A = (A_{\alpha,\beta})_{\alpha,\beta \in \Pi}$ is nonsingular 
(this is clear from inspection of the tables in \cite[Section 3.6]{Carter}; 
alternatively an inverse for $A$ in $\GL_{\Pi}(\mathbb{Q})$ 
is calculated explicitly in \cite[\textsection1.6]{Rosen}; 
see (1.157) and (1.158) therein). 
Putting $A$ into Smith Normal Form, 
there exist $P,Q \in \GL_{\Pi} (\mathbb{Z})$ 
and $0 \neq D_{\alpha} \in \mathbb{Z}$ such that: 
\begin{equation*}
(PAQ)_{\alpha,\beta} 
= \Big\{ \begin{array}{cc} D_{\alpha} & \alpha = \beta \\
0 & \alpha \neq \beta 
\end{array}
\end{equation*}
For $B \in \GL_{\Pi}(\mathbb{Z})$, define the morphism 
$\psi^{(B)} : (K^{\ast})^{\Pi} \rightarrow (K^{\ast})^{\Pi}$ by: 
\begin{equation*}
\psi^{(B)}(\underline{z})_{\alpha}=\prod_{\beta\in\Pi}z_{\beta} ^{B_{\beta,\alpha}}
\end{equation*}
so that for $B,C \in \GL_{\Pi}(\mathbb{Z})$, 
$\psi^{(BC)} = \psi^{(C)} \circ \psi^{(B)}$ 
and $\psi^{(I)} = \Id$. 
Then:
\begin{equation} \label{CartanVecEqn}
(\psi^{(Q)} \circ \psi^{(A)} \circ \psi^{(P)})(\underline{z})_{\alpha}
= z_{\alpha} ^{D_{\alpha}} 
\text{ and } \psi^{(A)} (\underline{z})_{\alpha} = \phi(\underline{z})_{\alpha}
\end{equation}
so if $\psi^{(P)} (\underline{z}) \in \ker (\phi)$, 
then $z_{\alpha}$ is a $D_{\alpha}$th root of unity, for all $\alpha \in \Pi$. 
However $\psi^{(P)}$ is invertible, with inverse $\psi^{(P^{-1})}$, so 
$\lvert \ker (\phi) \rvert \leq \prod_{\alpha \in \Pi} D_{\alpha}$. 
\end{proof}

\begin{propn} \label{ChevMaxTorusProp}
The image $T_0 = \im (\phi)$ of $\phi$ is a maximal torus in $\mathcal{L} (K)$. 
\end{propn}

\begin{proof}
$\im (\phi)$ is a closed connected abelian subgroup of $\mathcal{L} (K)$ 
(being the image of such under a morphism of algebraic groups). 
It consists of diagonal matrices, hence is a torus, 
so is contained in a maximal torus $T$ 
(of dimension $l$, the rank of $\mathcal{L} (K)$). 
Since $\ker (\phi)$ is finite, $\im (\phi) \subseteq T$ also 
has dimension $l$, hence they are equal. 
\end{proof}

Consider $T_0 (q) = T_0 \cap \mathcal{L} (\mathbb{F}_q)$. 
By (\ref{ChevtorusWordeqn}), 
$T_0 (q)$ contains $\phi(\underline{z})$ 
for every $\underline{z} \in (\mathbb{F}_q ^{\ast})^{\Pi}$, 
so $\lvert T_0 (q) \rvert \geq (q-1)^l / \lvert \ker (\phi) \rvert$. 
In the other direction, 
using the notation of Lemma \ref{morphfinkerlem}, 
suppose $\underline{z} \in (K^{\ast})^{\Pi}$ 
with $\phi (\underline{z}) \in T_0 (q)$. 
Writing $\underline{z} = \psi^{(P)} (\underline{y})$, 
by (\ref{CartanVecEqn}) we have 
$y_{\alpha} ^{D_{\alpha}} \in \mathbb{F}_q ^{\ast}$ for all $\alpha \in \Pi$, 
hence: 
\begin{equation} \label{torusorderineq}
(q-1)^l / \lvert \ker (\phi) \rvert 
\leq \lvert T_0 (q) \rvert 
\leq (q-1)^l \prod_{\alpha \in \Pi} D_{\alpha}
\end{equation}
Now, the roots of $\mathcal{L} (K)$ relative to $T_0 = \im(\phi)$ 
are indexed by the elements of $\Phi$: 
for each $\beta \in \Phi$ there is 
a morphism $T_0 \rightarrow K^{\ast}$ 
(which we also denote by $\beta$) given by: 
\begin{equation*}
\beta\big(\phi(\underline{z})\big)=\phi(\underline{z})_{\beta}
\end{equation*}
(see (\ref{charactereqn})). 
Thus $\langle e_{\beta} \rangle \subseteq \mathfrak{u}_{\beta}$, 
and since the maps $\beta : T_0 \rightarrow K^{\ast}$ 
are all distinct, nontrivial, and $T_0$ acts trivially on 
$\langle h_{\alpha} : \alpha \in \Pi \rangle$, we have equality. 
The root subgroup corresponding to $\beta$ 
is precisely $\lbrace x_{\beta}(\lambda) : \lambda \in K \rbrace$; 
in accordance with (\ref{RootActEqn}), 
\begin{equation*}
\phi(\underline{z}) x_{\beta}(\lambda) \phi(\underline{z})^{-1} 
= x_{\beta}\big( \beta\big(\phi(\underline{z})\big) \lambda \big)
\end{equation*}
(see \cite{Carter} \textsection 7.1). 

\begin{propn} \label{RegTorusOrderProp}
Let $R(q)=\lbrace s \in T_0 (q):s\text{ is regular in }\mathcal{L}(K) \rbrace$. 
Then: 
\begin{equation}
\lvert T_0 (q) \setminus R(q) \rvert = O_{\Phi} \big( \lvert T_0 (q) \rvert/q \big)\text{.}
\end{equation}
\end{propn}

\begin{proof}
Since $\Phi$ lies in the $\mathbb{R}$-span of $\Pi$, 
for every $\beta \in \Phi$ there exists $\alpha \in \Pi$ such that 
$A_{\alpha,\beta} \neq 0$. 
Then for $z \in \mathbb{F}_q ^{\ast}$, 
$\beta \big( \phi \big( (z^{\delta_{\alpha,\gamma}})_{\gamma \in \Pi} \big) \big)= z^{A_{\alpha,\beta}}$. 
Letting $N(\beta) = \hcf (\lbrace A_{\alpha,\beta} : \alpha \in \Pi \rbrace 
\setminus \lbrace 0 \rbrace )$, 
$\lvert \im \big(\beta|_{T_0(q)}  \big) \rvert \geq (q-1)/ N (\beta)$. 
By (\ref{torusorderineq}): 
\begin{equation*}
\lvert T_0 (q)\cap\ker (\beta)\rvert\leq (q-1)^{l-1} D N(\beta) 
\end{equation*}
where $D = \prod_{\alpha \in \Pi} D_{\alpha}$. 
Using the criterion for regularity from Proposition \ref{mainregeltpropn} (ii), 
and (\ref{torusorderineq}), we have: 
\begin{equation*}
\lvert T_0 (q) \setminus R(q) \rvert 
\leq (q-1)^{l-1} D \sum_{\beta\in\Phi} N(\beta) 
\text{ while } 
\lvert T_0(q) \rvert \geq (q-1)^l / \lvert \ker (\phi) \rvert
\end{equation*}
as desired. 
\end{proof}

\subsection{Proof of Proposition \ref{orderdivpropn}}

Let us recall the result, which it is our objective to prove here. 

\begin{propn} 
For $G$ a finite group and $m \in \mathbb{N}$, let: 
\begin{center}
$E_G (m) = \lbrace g \in G: o (g) \text{ divides } m \rbrace$. 
\end{center}
Let $G = X(q)$ be as in Theorem \ref{strongmainthm}. 
Define $b(X,q) \in \mathbb{N}$ to be: 
$q+1$ for $X = {^2} A_l$ or $^2 E_6$; 
$q^2 - 1$ for $X = {^2} D_l$; 
$q^2 - q + 1$ for $X = {^3} D_4$, 
or $q-1$ otherwise. Then: 
\begin{equation*}
\lvert E_G \big( b(X,q) \big) \rvert  = \Omega_X (\lvert G \rvert)\text{.}
\end{equation*}
\end{propn}

\begin{lem} \label{regcentlem}
Let $\Delta_n (q) \leq \GL_n (q)$ be the subgroup of diagonal matrices. 
Let $g \in \Delta_n (q)$ and suppose $g$ has no repeated eigenvalues. 
Then $C_{\GL_n (q)} (g) = \Delta_n (q)$. 
More generally let $g \in \GL_n (q)$ have the form: 
\begin{center}
$g = \left( \begin{array}{cc} I_m & 0 \\ 0 & h \end{array} \right)$
\end{center} 
for some $h \in \Delta_{n-m} (q)$ with no $1$-eigenvectors 
no repeated eigenvalues. Then any $c \in C_{\GL_n (q)} (g)$ 
has the form: 
\begin{equation*}
c = \left( \begin{array}{cc} k & 0 \\ 0 & d \end{array} \right) 
\end{equation*}
for some $k \in \GL_m (q)$ and $d \in \Delta_{n-m} (q)$. 
\end{lem}

\begin{lem} \label{orderdivquotlem}
Let $G$ be a finite group and let $N \vartriangleleft G$. 
Then for any $m \in \mathbb{N}$, 
\begin{center}
$\lvert E_{G/N}(m) \rvert / \lvert G/N \rvert 
\geq \lvert E_G(m) \rvert / \lvert G \rvert$
\end{center}
\end{lem}

\begin{proof}
If $g \in E_G(m)$, then $gN \in E_{G/N}(m)$, 
so $E_{G/N}(m)$ is the union of all cosets of $N$ intersecting $E_G(m)$. 
Thus $\lvert E_{G/N}(m) \rvert \geq \lvert E_G(m) \rvert / \lvert N \rvert$. 
\end{proof}

\begin{lem} \label{SchwZipLem}
Let $l,m \in \mathbb{N}$; 
for $1 \leq i \leq l$ and $1 \leq j \leq m$ 
let $a_i ^{(j)} \in \mathbb{Z}$, 
$\lambda^{(j)} \in \mathbb{F}_q ^{\ast}$ and: 
\begin{center}
$f^{(j)} (X_1,\ldots,X_l) 
= \lambda^{(j)} X_1 ^{a_1 ^{(j)}} \cdots X_l ^{a_l ^{(j)}} 
\in \mathbb{F}_q \lbrack X_1 ^{\pm 1} , \ldots , X_l ^{\pm 1} \rbrack$. 
\end{center}
Suppose that the $f^{(j)}$ are all distinct, 
and none is identically $1$. 
Let $S \subseteq \mathbb{F}_q$ and $D>0$, 
and suppose that for all $i$ and $j$, 
$\lvert a_i ^{(j)} \rvert \leq D$. 
Then there is a subset $A \subseteq S^l$ such that: 
\begin{itemize}
\item[(i)] $\lvert A \rvert \geq \big( \lvert S \rvert - Dm(m+1) \big)^l$; 
\item[(ii)] For all $\mathbf{a} \in A$, 
the $f^{(j)} (\mathbf{a}) \in \mathbb{F}_q$ are distinct and $\neq 1$. 
\end{itemize}
\end{lem}

\begin{proof}
We proceed by induction on $l$. 
If $l=1$ then for $s \in S$, 
if $f^{(j)} (s) = f^{(k)} (s)$ 
or $f^{(j)} (s) = 1$ 
for some $j\neq k$, 
then $s$ satisfies one of a set of $m(m+1)/2$ nontrivial polynomials 
of degree at most $2D$, 
so there are at most $Dm(m+1)$ possibilities for $s$. 

For $l \geq 2$, for at least $\lvert S \rvert - Dm(m+1)$ values $s_l \in S$, 
for any $1 \leq j,k \leq m$, 
$\lambda^{(j)} s_l ^{a_l ^{(j)}} = \lambda^{(k)} s_l ^{a_l ^{(k)}}$ 
implies $\lambda^{(j)} = \lambda^{(k)}$ and $a_l ^{(j)} = a_l ^{(k)}$ 
(by the preceding paragraph). 
Given such a value $s_l$, 
the Laurent polynomials: 
\begin{center}
$f^{(j)} (X_1,\ldots,X_{l-1},s_l) \in 
\mathbb{F}_q \lbrack X_1 ^{\pm 1} , \ldots , X_{l-1} ^{\pm 1} \rbrack$
\end{center}
are all distinct. 
Likewise for such values $s_l$, 
$\lambda^{(j)} s_l ^{a_l ^{(j)}} = 1$ 
implies $\lambda^{(j)} = 1$ and $a_l ^{(j)} = 0$, 
so for such values $s_l$, 
no $f^{(j)} (X_1,\ldots,X_{l-1},s_l)$ 
is identically $1$. 
\end{proof}

\begin{rmrk} \label{SchwZipRmrk}
In all cases in which we apply Lemma \ref{SchwZipLem}, 
either (i) $\lvert S \rvert = \Omega (q)$ and $D = O(\sqrt{q})$, 
or (ii) $\lvert S \rvert = \Omega (\sqrt{q})$ and $D$ is bounded. 
Under either of these conditions, it follows that 
for all $\epsilon > 0$, if $q = \Omega_{\epsilon,l,m} (1)$ then: 
\begin{equation} 
\label{SchwZipEqn}
\lvert A \rvert \geq (1-\epsilon) \lvert S \rvert^l\text{.}
\end{equation}
In particular, $\lvert A \rvert = \Omega_{l,m} (\lvert S \rvert ^l)$. 
\end{rmrk}

\subsubsection{Untwisted groups}

Here $G = \mathcal{L} (\mathbb{F}_q)$ 
is as in Definition \ref{UntwistedGrpDefn}. 
Let $T_0$ be as in Proposition \ref{ChevMaxTorusProp}; 
let $T_0 (q) = T_0 \cap G$, 
and let $R(q) \subseteq T_0 (q)$ be the set of regular elements. 
Note that any conjugate in $G$ of an element of $T_0 (q)$ 
has order dividing $q-1$, 
so the conclusion of Proposition \ref{orderdivpropn} 
reduces to the following. 

\begin{propn}
Let $G$, $R(q)$ be as above. Then: 
\begin{equation} \label{RegCCLEqn}
\big\lvert \bigcup_{r \in R(q)} \ccl_G (r) \big\rvert 
= \Omega_{\Phi} (\lvert G \rvert)\text{.}
\end{equation}
\end{propn}

\begin{proof}
If $r_1 , r_2 \in R(q)$ and $g_1 , g_2 \in G$ are such that 
$r_1 ^{g_1} = r_2 ^{g_2}$, 
then by Proposition \ref{RegTorusNormCoroll} $g_1 g_2 ^{-1} \in N_G (T_0)$ 
so the fibres of the map sending 
$(r,g) \in R(q) \times G$ to $r^g$ are of size bounded by 
$\lvert N_G (T_0) \rvert \leq \lvert W \rvert \cdot \lvert T_0 (q) \rvert$ 
(by Theorem \ref{WeylThm}). Thus: 
\begin{equation*}
\big\lvert \bigcup_{r \in R(q)} \ccl_G (r) \big\rvert \geq \frac{\lvert R(q) \rvert}{\lvert W \rvert \cdot \lvert T_0 (q) \rvert} \lvert G \rvert
\end{equation*}
and the result follows, since by Proposition \ref{RegTorusOrderProp}, 
$\lvert R(q) \rvert / \lvert T_0 (q) \rvert = \Omega_{\Phi} (1)$. 
\end{proof}

\begin{rmrk} \label{twistedorderrmrk}
The bound (\ref{RegCCLEqn}) establishes Proposition \ref{orderdivpropn} 
for the groups $X(q)$ with $X=A_l$, $B_l$, $C_l$, $D_l$, 
$E_6$, $E_7$, $E_8$, $F_4$ or $G_2$. 
The basic strategy for the other families will be the same: 
in each case we produce $R(q) \subseteq G = X(q)$ such that: 
(a) for all $r \in R(q)$, $o(r)$ divides $b(X,q)$; 
(b) the fibres of the map sending 
$(r,g)\in R(q) \times G$ to $r^g$ are of size $O_X\big( \lvert R(q)\rvert \big)$. 
Note in particular that (b) holds whenever, for all $r \in R(q)$, 
(b)(i) $\lvert \ccl_G(r) \cap R(q) \rvert = O_X (1)$, 
and (b)(ii) $\lvert C_G (r) \rvert = O_X (\lvert R(q) \rvert)$. 
As above we have a bound of the form (\ref{RegCCLEqn}), 
from which the result follows. 
\end{rmrk}

\subsubsection{Unitary groups}

For ${^2} A_l (q) = \PSU_{l+1} (q)$, 
using Lemma \ref{orderdivquotlem} it suffices 
to show that $G = \SU_{l+1} (q)$ satisfies 
$\lvert E_G (q+1)\rvert = \Omega_l (\lvert G \rvert)$. 
Recall that $\GU_{l+1} (q)$ is the subgroup of the group $\GL_{l+1} (q^2)$ 
preserving a non-degenerate Hermitian form 
on $\mathbb{F}_{q^2} ^{l+1}$. 
Up to similarity, there is only one such form, given by $I_{l+1}$. 
Given $\underline{\lambda} \in \mathbb{F}_{q^2} ^{l+1}$, 
$\diag(\underline{\lambda}) \in \GU_{l+1} (q)$ iff 
$\lambda_1 , \ldots , \lambda_{l+1} \in S$, 
the set of $(q+1)$th roots of unity in $\mathbb{F}_{q^2}$. 
For $\diag(\underline{\lambda}) \in \GU_{l+1} (q)$ iff: 
\begin{center}
$I_{l+1} = \phi (\diag(\underline{\lambda}))^t \diag(\underline{\lambda})
 = \diag(\underline{\lambda})^{q+1}$
\end{center}
where $\phi : \GL_{l+1} (q^2) \rightarrow \GL_{l+1} (q^2)$ 
is the Frobenius automorphism given by $\phi(g)_{i,j} = g_{i,j} ^q$. 

Thus $\diag(\underline{\lambda}) \in \SU_{l+1} (q)$ 
iff $\lambda_1 , \ldots , \lambda_l \in S$ 
and $\lambda_{l+1} = (\lambda_1 \cdots \lambda_l)^{-1}$. 
We apply Lemma \ref{SchwZipLem} with $m=l+1$; 
$S$ as above; $D=1$, and: 
\begin{center}
$f^{(1)}=X_1,\ldots ,f^{(l)}=X_l,f^{(l+1)}=(X_1 \cdots X_l)^{-1}$. 
\end{center}
Using Remark \ref{SchwZipRmrk}, we see that the set 
$R(q) \subseteq \SU_{l+1} (q)$ of diagonal matrices 
without repeated eigenvalues has size $\Omega_{l} \big( (q+1)^l \big)$. 
We check the conditions of Remark \ref{twistedorderrmrk}. 
Certainly any element of $R(q)$ 
has order dividing $q+1$. 
Since any two conjugate matrices have the same eigenvalues we have (b)(i), 
and by Lemma \ref{regcentlem} the centralizer of any $r \in R(q)$ 
consists of diagonal elements of $\SU_{l+1}(q)$ (of which there are $(q+1)^{l}$), 
whence (b)(ii). 
Finally we pass from $\SU_{l+1}(q)$ to 
$\PSU_{l+1}(q)$ using Lemma \ref{orderdivquotlem}. 

\subsubsection{Orthogonal groups of type ${^2}D_l$}

Note that $^2 D_l(q)= P \Omega_n ^- (q)$, 
where $n = 2l$. 
By Lemma \ref{orderdivquotlem} it suffices to show that 
$G=\Omega_n ^- (q)$ satisfies $\lvert E_G(q^2 -1) \rvert =\Omega_n (\lvert G\rvert)$. 
We start by recalling some background information 
on the forms preserved by these groups, taken from \cite{Carter}. 
If $q$ is odd then the orthogonal group 
${\rm O}_n ^- (q)$ is the group of isometries 
of the symmetric bilinear form $B^-$ 
on $\mathbb{F}_q ^n$, 
represented by: 
\begin{center}
$\left( \begin{array}{cccc} 
 0 & I_{l-1} & 0 & 0 \\
I_{l-1} & 0 & 0 & 0 \\
0 & 0 & 1 & 0 \\
0 & 0 & 0 & -\gamma \end{array} \right)
 \in \mathbb{M}_{n} (\mathbb{F}_q)$ 
\end{center}
with respect to the standard ordered basis 
$\mathcal{B} = \lbrace b_i \rbrace_{i=1} ^n$ 
of $\mathbb{F}_q ^n$, 
where $\gamma \in \mathbb{F}_q$ is a non-square. 
If on the other hand $q$ is even then 
${\rm O}_n ^- (q)$ is the group of isometries 
of the quadratic form on $\mathbb{F}_q ^n$, 
represented by: 
\begin{align*}
f^- (\underline{x}) & = \sum_{i=1} ^{l-1} x_{i} x_{l-1+i} 
+ (x_{2l-1} - \alpha x_{2l})(x_{2l-1} - \overline{\alpha} x_{2l})
\end{align*}
where $\alpha$ generates $\mathbb{F}_{q^2}$ over $\mathbb{F}_q$ 
(note that $f^-$ is nonetheless defined over $\mathbb{F}_q$). 
In this case, let $B^{-}$ be the \emph{polar form} 
of a quadratic form $f^{-}$, given by: 
\begin{center}
$B^{-}(\underline{x},\underline{y})=f^{-}(\underline{x}+\underline{y})
-f^{-}(\underline{x})-f^{-}(\underline{y})$. 
\end{center}
Then $B^{-}$ is a symmetric bilinear form 
preserved by ${\rm O}_n ^{-} (q)$; its matrix 
with respect to $\mathcal{B}$ is: 
\begin{center}
$\left( \begin{array}{cccc} 
 0 & I_{l-1} & 0 & 0 \\
I_{l-1} & 0 & 0 & 0 \\
0 & 0 & 0 & \beta \\
0 & 0 & \beta & 0 \end{array} \right)
 \in \mathbb{M}_{n} (\mathbb{F}_q)$ 
\end{center}
where 
$\beta = -(\alpha + \overline{\alpha}) \in \mathbb{F}_q$. 
In all cases, $\Omega_n ^-(q)$ is a subgroup 
of index $2$ in $\SO_n ^-(q)$, see for instance \cite{BrHoRD} Definition 1.6.13 
for a precise description of $\Omega_n ^-(q)$ as a subgroup 
of $\SO_n ^-(q)$. 

The group $\SO^{-}_2(q)$ is cyclic of order $q+1$; 
let $\zeta\in \SO^{-}_2(q)$ be a generator. 
Since $\hcf(q+1,q-1) \leq 2$, $\zeta^m$ 
is not diagonalizable over $\mathbb{F}_q$ 
for any $1 \leq m \leq q$ 
with $m \neq (q+1)/2$. 

Now let $D(q) \leq \SO_{2l} ^- (q)$ 
be the subgroup of elements of the form: 
\begin{center}
$\left( \begin{array}{cc} 
 d(\underline{\lambda}) & 0 \\
 0 & \zeta^m \end{array} \right)$, 
where 
$d(\underline{\lambda}) = \diag(\lambda_1 , \ldots , \lambda_{l-1} , \lambda_1 ^{-1} , \ldots , \lambda_{l-1} ^{-1}) \in \Delta_{n-2} (q)$ 
\end{center}
for some $\underline{\lambda} = (\lambda_1 , \ldots , \lambda_{l-1}) \in (\mathbb{F}_q ^{\ast})^{l-1}$, 
and let $R(q) \subseteq D(q)$ be the set of 
elements such that 
$1 \neq \lambda_1 , \ldots , \lambda_{l-1} , \lambda_1 ^{-1} , \ldots , \lambda_{l-1} ^{-1}$ are all distinct and $1 \leq m \leq q$ with $m \neq (q+1)/2$. 
Applying Lemma \ref{SchwZipLem} 
to $X_1 ^{\pm 1} , \ldots , X_{l-1} ^{\pm 1}$, 
we have from (\ref{SchwZipEqn}) that 
$\lvert R(q) \rvert \geq 3 \lvert D(q) \rvert / 4$ 
for $q$ larger than an (explicit) absolute constant, 
so $\lvert R(q) \cap \Omega_n ^- (q) \rvert \geq  \lvert D(q) \cap \Omega_n ^- (q) \rvert / 2$. 
Let: 
\begin{center}
$r_1 = \left( \begin{array}{cc} 
 d(\underline{\lambda}) & 0 \\
 0 & \zeta^m \end{array} \right), 
 r_2 = \left( \begin{array}{cc} 
 d(\underline{\mu}) & 0 \\
 0 & \zeta^k \end{array} \right) \in R(q) \cap \Omega_n ^- (q)$
\end{center}
and suppose $r_1$ and $r_2$ are conjugate in 
$\Omega_n ^- (q)$. 
The eigenvalues of $\zeta$ in $\mathbb{F}_{q^2} ^{\ast}$ are $\omega^{\pm 1}$, for some 
$\omega \in \mathbb{F}_{q^2} ^{\ast}$ of order $q+1$. 
Since $k,m \neq (q+1)/2$, $\omega^{\pm m}$ 
and $\omega^{\pm k}$ do not lie in $\mathbb{F}_q$, 
hence $\lbrace \omega^{\pm m} \rbrace = \lbrace \omega^{\pm k} \rbrace$, 
so given $m$ there are only $2$ possibilities for $k$. 
The (common) eigenvalues of $r_1$ and $r_2$ 
in $\mathbb{F}_q$ 
are the (common) entries of $\underline{\lambda}$ 
and $\underline{\mu}$. Thus for fixed 
$\underline{\lambda}$, 
there are at most $2^{l-1} (l-1)!$ possibilities 
for $\underline{\mu}$. 
Thus 
$\lvert \ccl_{\Omega_n ^- (q)} (r_1) \cap R(q) \rvert \leq 2^l (l-1)!$. 

Now suppose $r \in R(q) \cap \Omega_n ^- (q)$ and let 
$g \in C_{\SO^{-} _n (q)}(r)$. 
Then $g$ preserves each eigenspace $\langle b_i \rangle$ 
for $1 \leq i \leq n-2$, 
hence preserves the orthogonal complement 
of $W = \langle b_1 , \ldots , b_{n-2} \rangle$ under $B^{-}$, namely $U = \langle b_{n-1},b_n \rangle$. 
The fact that $g$ preserves $B^-$ 
(and preserves $f^- \mid_{U}$ 
in the case of $q$ even) implies 
that $g$ has the form: 
\begin{center}
$\left( \begin{array}{cc} 
 d(\underline{\mu}) & 0 \\
 0 & h \end{array} \right)$ 
 for some $\underline{\mu} \in (\mathbb{F}_q ^{\ast})^{l-1}$ and $h \in \SO_2 ^{-} (q)$
\end{center}
so that $\lvert C_{\Omega^{-} _n(q)}(r) \rvert \leq 
\lvert C_{\SO^{-} _n (q)}(r) \rvert \leq 2 \lvert D(q) \rvert \leq 8 \lvert R(q) \cap \Omega_n ^- (q) \rvert$. 
We have satisfied the conditions of Remark \ref{twistedorderrmrk}. 

\subsubsection{Suzuki groups}

For the groups ${^2}B_2(q)$ we use the 
$4$-dimensional linear representation over $\mathbb{F}_q$ 
described in \cite{Wilson} \textsection 4.2. 
Here $q = 2^{2n+1}$. 
The form of the elements of ${^2}B_2(q) \cap \Delta_4 (q)$ 
is described in (4.10) from \cite{Wilson} p.115; 
they are: 
\begin{center}
$\lbrace \diag (\alpha,\alpha^{2^{n+1} - 1},\alpha^{-2^{n+1} + 1},\alpha^{-1}) : \alpha \in \mathbb{F}_q ^{\ast} \rbrace$
\end{center}
(so that $\lvert {^2}B_2(q) \cap \Delta_4 (q) \rvert = q-1$). 
Applying Lemma \ref{SchwZipLem} with 
$l=1$; $m=4$; $S = \mathbb{F}_q ^{\ast}$; $D = 4 \sqrt{q}$, and: 
\begin{center}
$f^{(1)}(X_1)=X_1, f^{(2)}(X_1)=X_1 ^{2^{n+1} - 1}, 
f^{(3)}(X_1) = X_1 ^{-2^{n+1} + 1}, f^{(4)}(X_1) = X_1 ^{-1}$, 
\end{center}
we conclude that there exists $R (q) \subseteq {^2}B_2(q) \cap \Delta_4 (q)$ 
consisting of elements without repeated eigenvalues 
and $\lvert R (q) \rvert \geq q/2 \geq \lvert {^2}B_2(q)\cap\Delta_4 (q)\rvert / 2$. 
As before, any element of $R(q)$ 
has order dividing $q-1$. 
Given $r \in R(q)$, 
$\lvert \ccl_{{^2}B_2(q)}(r) \cap R(q) \rvert 
\leq \lvert \ccl_{{^2}B_2(q)}(r) \cap \Delta_4(q) \rvert\leq 24$ , since any two conjugate elements have the same eigenvalues, 
and by Lemma \ref{regcentlem} 
$C_{{^2}B_2(q)}(r) \leq {^2}B_2(q)\cap\Delta_4 (q)$. 
Thus $R(q)$ satisfies Remark \ref{twistedorderrmrk}. 

\subsubsection{Small Ree groups}
The argument for ${^2} G_2 (q)$ is essentially identical 
(with $q = 3^{2n+1}$, using the $7$-dimensional linear representation 
explained in \cite{Wilson} \textsection 4.5; 
see in particular the description of the diagonal matrices 
lying in ${^2} G_2 (q)$, given in (4.53) from \cite{Wilson} p.137). 

\subsubsection{Large Ree groups}

For ${^2} F_4 (q)$ we use the model described in \cite[Section 4.9]{Wilson}. 
First consider the group $F_4 (q)$. 
As described in \cite[Section 4.8]{Wilson} this group 
has a faithful $27$-dimensional (reducible) linear representation 
on a nonassociative $\mathbb{F}_q$-algebra $\mathcal{A}$. 
There is a preferred $\mathbb{F}_q$-basis 
$\mathcal{W} = \lbrace w_i,w_i ^{\prime},w_i ^{\prime\prime}\rbrace_{i=0} ^8$
for $\mathcal{A}$ defined in \cite[Subsection 4.8.4]{Wilson}. 
Let $\circ$ denote the product on $\mathcal{A}$. 
Then $\circ$ is commutative and satisfies the following 
relations (among others): 
\begin{equation} \label{Albertprodeqn}
\begin{array}{ccc}
w_0 \circ w_1 = 0; & 
w_0 \circ w_1^{\prime} = w_1^{\prime}; & 
w_0 \circ w_1^{\prime\prime} = w_1^{\prime\prime}; \\
w_0 ^{\prime} \circ w_1 = w_1; &
w_0 ^{\prime} \circ w_1 ^{\prime} = 0; &
w_0 ^{\prime} \circ w_1^{\prime\prime} = w_1^{\prime\prime};  \\
w_0 ^{\prime\prime} \circ w_1 = w_1;  &
w_0 ^{\prime\prime} \circ w_1 ^{\prime} = w_1 ^{\prime\prime}; &
w_0 ^{\prime\prime} \circ w_1 ^{\prime\prime} = 0; \\
w_1 \circ w_8 = w_0 ^{\prime} + w_0 ^{\prime\prime}; & 
w_1 ^{\prime} \circ w_8 ^{\prime} = w_0 + w_0 ^{\prime\prime}; & 
w_1 ^{\prime\prime} \circ w_8 ^{\prime\prime} = w_0 + w_0 ^{\prime}.  \\
\end{array}
\end{equation}
Moreover $\circ$ is preserved by the action of $F_4 (q)$. 

Let $q = 2^{2n+1}$; identify $F_4 (q)$ with its image in $\GL (\mathcal{A})$ 
and embed $^2 F_4 (q)$ as a subgroup of $F_4 (q)$ as in 
\cite[Subsection 4.9.1]{Wilson}. 
The elements of $^2 F_4 (q)$ which are diagonal with 
respect to the basis $\mathcal{W}$ may be parametrised 
as $\lbrace g_{(\alpha , \beta)} \rbrace_{\alpha , \beta 
\in \mathbb{F}_q ^{\ast}}$, where 
the eigenvalues of $g_{(\alpha , \beta)}$ 
on the basis $\mathcal{W}$
are as given in \cite[Subsection 4.9.2]{Wilson}. 
The vectors $w_0, w_0 ^{\prime} , w_0 ^{\prime\prime}$ 
(called $w_9, w_9 ^{\prime} , w_9 ^{\prime\prime}$ 
in \cite[Subsection 4.9.2]{Wilson}, 
in contrast to the notation in \cite[Section 4.8]{Wilson}) 
are $1$-eigenvectors of $g_{(\alpha , \beta)}$, 
and there are Laurent polynomials 
$f^{(1)} , \ldots , f^{(24)} 
\in \mathbb{F}_q \lbrack X_1 ^{\pm 1},X_2 ^{\pm 1} \rbrack$, 
such that the remaining elements of $\mathcal{W}$ 
are eigenvectors of $g_{(\alpha , \beta)}$ with eigenvalues 
$f^{(1)}(\alpha , \beta) , \ldots , f^{(24)}(\alpha , \beta)$ 
(see the table in \cite[Subsection 4.9.2]{Wilson}). 
Moreover the family $1 , f^{(1)} , \ldots , f^{(24)}$ 
satisfies Lemma \ref{SchwZipLem} 
(with $l=2$, $m=24$, $S = \mathbb{F}_q ^{\ast}$ and $D = 4 \sqrt{q}$). 
By Lemma \ref{SchwZipLem} and Remark \ref{SchwZipRmrk}, 
there exists a set $R(q) \subseteq {^2} F_4 (q)$, 
with $\lvert R(q) \rvert \geq q^2 / 2$ 
for $q$ larger than an explicit absolute constant, 
such that $R(q)$ consists of diagonal matrices 
(with respect to the basis $\mathcal{W}$) 
of the form: 
\begin{center}
$g = \left( \begin{array}{cc} I_3 & 0 \\ 0 & h \end{array} \right)$
\end{center} 
for some $h \in \Delta_{24} (q)$ with no $1$-eigenvectors no repeated eigenvalues. 
The set $R(q)$ satisfies conditions (a) 
and (b)(i) of Remark \ref{twistedorderrmrk} 
as in the previous cases. 
Moreover by Lemma \ref{regcentlem} 
any $c \in C_{({^2}F_4 (q))} (g)$ has the form: 
\begin{equation*}
c = \left( \begin{array}{cc} k & 0 \\ 0 & d \end{array} \right) 
\end{equation*}
for some $k \in \GL_3 (q)$ and $d \in \Delta_{24} (q)$. 
The fibres of the projection of $C_{({^2}F_4 (q))} (g)$ 
onto the top-left $3 \times 3$ block have 
order bounded by $\lvert  \rvert = O(q^2)$, 
so to verify condition 
(b)(ii) of Remark \ref{twistedorderrmrk}, 
it suffices to bound the image of this projection map, 
that is, to bound the number of possibilities for $k$. 

Using the fact that the action of $^2 F_4 (q)$ 
preserves``$\circ$'', the first three rows of 
relations from (\ref{Albertprodeqn}) yield that: 
 \begin{equation} \label{Albertprodeqn}
\begin{array}{ccc}
 k_{1,1}+k_{2,1}=1;  & k_{2,1}+k_{3,1}=0;  & k_{3,1}+k_{1,1}=1; \\
 k_{1,2}+k_{2,2}=1;  & k_{2,2}+k_{3,2}=1;  & k_{3,2}+k_{1,2}=0; \\
k_{1,3}+k_{2,3}=0; & k_{2,3}+k_{3,3}=1; & k_{3,3}+k_{1,3}=1
\end{array}
\end{equation}
so that $k$ has the form: 
\begin{center}
$k = I_3 + \left( \begin{array}{ccc} x & y & z \\ x & y & z \\ x & y & z \end{array} \right)$. 
\end{center} 
However the final row of 
relations from (\ref{Albertprodeqn}) then shows 
that $x+y=y+z=z+x=0$, so $x=y=z$. 
Finally, since the eigenvalues of $d$ occur 
in inverse pairs (see \cite[Subsection 4.9.2]{Wilson}), 
we have $\det(k)=1$, so there are only three possibilities for $x$, and hence for $k$. 

\subsubsection{Type ${^2}E_6$}

The group $E_6 (q)$ admits a central extension $\SE_6 (q)$ of degree dividing $3$. 
As discussed in \cite[Section 4.10]{Wilson}, 
$\SE_6 (q)$ admits a faithful $27$-dimensional linear representation 
over $\mathbb{F}_q$. 
A preferred basis $\mathcal{W} = \lbrace w_0,\ldots,w_8 ^{\prime\prime} \rbrace$ 
is given, and we can define a Hermitian form, 
with respect to which $\mathcal{W}$ is orthonormal. 
The group ${^2}E_6 (q)$ 
is the image in  $E_6 (q^2)$ of the subgroup 
$H(q)$ of $\SE_6 (q^2)$ consisting of those elements 
preserving this Hermitian form 
(see \cite[Section 4.11]{Wilson}). 
By Lemma \ref{orderdivquotlem} it suffices to show that 
$\lvert E_{H(q)} (q+1)\rvert = \Omega (\lvert H(q) \rvert)$. 

There is described in \cite[Subsection 4.10.3]{Wilson} 
a subgroup $T=T(q)$ of $\SE_6 (q)$, 
which is the intersection of $\SE_6 (q)$ with a maximal torus of the simple linear algebraic group $\SE_6$ 
(defined over $K = \overline{\mathbb{F}_q}$). 
The elements of $T(q)$ are diagonal with 
respect to the basis $\mathcal{W}$, 
and are parametrized by six elements 
$\alpha , \beta , \gamma , \delta , \lambda , \mu \in \mathbb{F}_q ^{\ast}$. 
Arguing as in the case of unitary groups, 
an element of $\SE_6 (q^2)$ lying in $T(q^2)$ 
is in $H(q)$ iff it is supported on the set $S$ 
of $(q+1)$th roots of unity. 
We apply Lemma \ref{SchwZipLem} to this set $S$ 
(with $l=6$; $m=27$ and $D=2$) 
and the $27$ Laurent polynomials described in the table 
in \cite[Subsection 4.10.3]{Wilson}. 
We conclude that there exists a 
set $R(q) \subseteq T(q^2) \cap H(q)$, 
such that $\lvert R(q) \rvert \geq q^6 / 2$ 
for $q$ larger than an explicit absolute constant, 
and $R(q)$ consists of elements of 
$T(q^2) \cap H(q)$, all of whose eigenvalues 
are distinct from each other and distinct from $1$. 
Conditions (a) and (b)(i) 
of Remark \ref{twistedorderrmrk} 
are then clear, as in the previous cases. 
The centralizer $C_{H(q)} (g)$ of an element 
$g \in R(q)$ is diagonal by Lemma \ref{regcentlem}. 
It is not self-evident that $C_{H(q)} (g)$ 
is contained in $T(q^2) \cap H(q)$, 
however modulo scalars $C_{H(q)} (g)$ 
lies in the normalizer $N$ of $T(q^2)$ in $E_6 (q^2)$. 
The quotient map $C_{H(q)} (g) \rightarrow N/T(q^2)$ 
has fibres of size bounded by 
$\lvert T(q^2) \cap H(q) \rvert (q+1)^6$, 
and image a subgroup of $W(E_6) \cong \SO_5 (3)$ 
(see \cite[Subsection 4.10.3]{Wilson} again), 
and we have condition (b)(ii) 
of Remark \ref{twistedorderrmrk}. 


\subsubsection{Type ${^3}D_4$}

The conclusion of Proposition \ref{orderdivpropn} for the groups 
${^3}D_4 (q)$ is essentially contained in \cite{DeriMich}. 
Therein, $G$ is a simple, simply-connected group of Dynkin type $D_4$ 
over $K = \overline{\mathbb{F}_p}$; $q$ is a power of $p$; 
$T$ is a maximal torus of $G$ and $\sigma$ is an automorphism of $G$ 
such that $T$ is $\sigma$-stable and ${^3}D_4 (q) = G_{\sigma}$, 
the set of $\sigma$-fixed points of $G$. 

There is another $\sigma$-stable maximal torus $T^{\prime}$ of $G$, 
and $g \in G$ such that $(T^{\prime} _{\sigma})^g =  T_4 \subseteq T$, 
where $T_4 \cong C_{q^2 - q + 1 } ^2$ 
is as described in Table 1.1 from \cite{DeriMich}. 
The regular semisimple elements of $G$ lying in $T_4$ 
are described in Propositions 2.1-2.2 of \cite{DeriMich}: 
they are the elements of the form $s_{13}$ appearing in Table 2.1 
from that paper. 

By Table 4.4 from \cite{DeriMich} and the discussion preceding it, 
the number of conjugacy classes of elements of type $s_{13}$ 
in $G_{\sigma}$ (which is the same as the number of 
irreducible characters of type $\chi_{13}$) is 
$(q^4 - 2 q^3 - q^2 + 2 q)/24$. 
It therefore suffices to show that each such element $x$ 
has centralizer in $G_{\sigma}$ of order $O(q^4)$. 
This is so, because $x$ is regular: 
by Proposition \ref{RegTorusNormCoroll}, 
$C_{G_{\sigma}} (x) \leq N_{G_{\sigma}} (T^{\prime})$, so: 
\begin{align*}
\lvert C_{G_{\sigma}} (x) \rvert & \leq \lvert N_{G_{\sigma}} (T^{\prime}):T^{\prime} _{\sigma}\rvert\cdot\lvert T^{\prime} _{\sigma}\rvert \\ 
& \leq \lvert N_G(T^{\prime}):T^{\prime} \rvert \cdot \lvert T_4 \rvert \\
& \leq \lvert W \rvert (q^2-q+1)^2 \text{ (by Theorem \ref{WeylThm})}
\end{align*}
as required. 

This concludes the proof of Proposition \ref{orderdivpropn}.

\Addresses

\end{document}